\documentclass[preprint,12pt]{elsarticle}
\usepackage{amsmath}
\usepackage{amsfonts}
\usepackage{bm}
\usepackage{lipsum}
\usepackage{enumerate}
\usepackage{graphicx}
\usepackage{amsfonts}

\usepackage{mathrsfs}
\usepackage{subfigure}
\usepackage{epstopdf}
\usepackage{url}
\usepackage{verbatim}
\usepackage{amssymb}
\usepackage{hyperref}
\usepackage{color}
\usepackage{listings}

\def\@begintheorem#1#2{\par\bgroup{\scshape #1\ #2. }\it\ignorespaces}
\def\@opargbegintheorem#1#2#3{\par\bgroup%
   {\scshape #1\ #2\ ({\upshape #3}). }\it\ignorespaces}
\def\@endtheorem{\egroup}

\if@onethmnum
  \newtheorem{theorem}{Theorem}
  \newtheorem{lemma}[theorem]{Lemma}
  \newtheorem{corollary}[theorem]{Corollary}
  \newtheorem{proposition}[theorem]{Proposition}
  \newtheorem{definition}[theorem]{Definition}
\newtheorem{example}[theorem]{Example}
\newtheorem{remark}[theorem]{Remark}
\newtheorem{homework}[theorem]{Homework}
\newtheorem{case}[theorem]{}

\else
  \newtheorem{theorem}{Theorem}[section]

\newtheorem{remark}[theorem]{Remark}

\fi

\usepackage[T1]{fontenc}

\begin{document}

\begin{frontmatter}



\title{Ergodicity and invariant measures for a diffusing passive scalar advected by a random channel shear flow and the connection between the Kraichnan-Majda model and Taylor-Aris Dispersion}

\author[1]{Lingyun Ding}
\ead{dingly@live.unc.edu}
\author[1]{Richard M. McLaughlin \corref{mycorrespondingauthor}}
\cortext[mycorrespondingauthor]{Corresponding author}
\ead{rmm@email.unc.edu}
\address[1]{Department of Mathematics, University of North Carolina, Chapel Hill, NC, 27599, United States}





\begin{abstract}
  We study the long time behavior of an advection-diffusion equation with a random shear flow which depends on a stationary Ornstein-Uhlenbeck (OU) process in  parallel-plate channels enforcing the no-flux boundary conditions.  We derive a closed form formula for the long time asymptotics of  the arbitrary $N$-point correlator using the ground state eigenvalue perturbation approach proposed in  \cite{bronski1997scalar}.  In turn, appealing to the conclusion of the Hausdorff moment problem \cite{shohat1943problem}, we discover a diffusion equation with a random drift and deterministic enhanced diffusion possessing the exact same probability distribution function at long times.  Such equations enjoy many ergodic properties which immediately translate to ergodicity results for the original problem.  In particular, we establish that the first two Aris moments using a single realization of the random field can be used to explicitly construct all ensemble averaged moments.  Also, the first two ensemble averaged moments explicitly predict any long time centered Aris moment.  Our formulae quantitatively depict the dependence of the deterministic effective diffusion on the interaction between spatial structure of flow and random temporal fluctuation.
  Further, this approximation provides many identities regarding the stationary OU process dependent time integral. We derive explicit formulae for the decaying passive scalar's long time limiting probability distribution function (PDF) for different types of  initial conditions (e.g. deterministic and random). 

\end{abstract}



\end{frontmatter}

\section{Introduction}

Passive scalars are extremely important quantities in many physical and biological applications including contamination in groundwater, solute transport in micro-fluidics, and even in the analysis of functional MRI brain scans.  Additionally they help to provide a basis for understanding problems in fluid turbulence.  For example, the $k^{-1}$ small scale power spectrum a scalar field inherits from a turbulent flow, as predicted by Batchelor \cite{bronski1997scalar}, has recently been rigorously established in a passive scalar model with velocities taken from randomly driven Navier-Stokes equations \cite{bedrossian2019batchelor}.  Moreover, they provide insight into intermittency in fluid turbulence whereby higher statistical moments deviate strongly from Gaussianity. 

An important class of problems concerns how a shear flow in a bounded (or partially bounded) cross-sectional domain can increase solute mixing. G. I. Taylor \cite{taylor1953dispersion} first showed that a steady pressure driven flow in a pipe leads to a greatly enhanced effective diffusivity for large  P\'{e}clet numbers. The analysis was later generalized by R. Aris for arbitrary spatial moments\cite{aris1956dispersion}. The dispersion process is sometimes also referred to as the Taylor-Aris dispersion. The recent studies  \cite{vedel2012transient,vedel2014time,ding2020enhanced} explored the case of a periodic time-varying shear flow and developed formulas of effective diffusivity. These studies are focused upon deterministic flows. The case involving random flows has additionally received great attention, particularly in understanding scalar intermittency. 

In a turbulent flow the distributions of most passive scalars such as pressure, temperature, concentration are generally far from Gaussian \cite{majda1999simplified,castaing1989scaling,belmonte1996thermal}. Even for roughly Gaussian velocity fields as observed in turbulent flows, rare fluctuations in amplitude have a significant contribution to non-Gaussianity in a scalar's distribution\cite{monin2013statistical}. Since the detailed structure of turbulence is still poorly understood, attempts to understand intermittency phenomenon have been explored in passive scalar models.  The most popular model used for this purpose is Kraichnan model \cite{kraichnan1968small} and the Majda model \cite{majda1993random,mclaughlin1996explicit,majda1999simplified} of passive scalar advection, where the random velocity field is assumed to be short correlated in time with a coherent (linear) structure in space.  The rapid correlation in time, particularly for multiplicative noise, allows for explicit ensemble moment closure through which explicit closed partial differential equations govern the generic $N$ point correlator (generally in $3N$ spatial dimensions).  Interestingly, Majda demonstrated for the case of a linear shear multiplied by a white in time Gaussian process in free space that these closed PDEs can be explicitly solved, and all moments explicitly computed showing how a heavy tailed scalar PDF is inherited from a Gaussian random field \cite{majda1993explicit,mclaughlin1996explicit,bronski2000rigorous,bronski2000problem}.   

Most prior studies of those models have focussed on the free space domain, fewer studies have addressed the effect of the physical boundary.  One such study contrasted the scalar PDF inherited by an unbounded linear shear with that of a bounded, periodic shear flow \cite{bronski1997scalar}.  This established that for integrable random initial data the PDF would `Gaussianize' at long times, whereas short ranged, random wave initial data would produce divergent flatness factors in the same limit as finite times.  Recently, the role of impermeable boundaries in the Majda model has been explored in a parallel-plate channel with deterministic initial conditions \cite{camassa2019symmetry,camassa2020persisting}.   Those works demonstrate that the sign of long time PDF skewness could be controlled by  P\'{e}clet numbers and the correlation time of the velocity field, in strong contrast with the free space result, where the long time PDF skewnness is strictly positive. \cite{mclaughlin1996explicit}. 
  
In this paper, we study a novel connection between ensemble (scalar intermittency) and spatial averaging (shear dispersion), arriving at a link between the Kraichnan-Majda model and Taylor-Aris Dispersion.  This link will yield a complete theory for the long time invariant measure of a diffusing passive scalar advected by an OU process dependent random shear flow.  Further, we will establish for the first time ergodicity results connecting certain spatial averages of the random passive scalar to the ensemble averages of the random passive scalar.  These ergodicity results are possible through an approximating effective advection-diffusion equation with a random drift we identified which possesses the exact same long time moments as the full problem.  Such random advection-diffusion equations were studied in \cite{bronski2007explicit}.  Strong ergodicity in which time averages of fields depending upon a single realization of the random process converge to the ensemble average are highly desirable, and provide connection between real experiments and theories developed using ensemble averaging.  Our results here will establish how the commonly measured effective diffusivity in mixing experiments is in fact ergodic in this model, converging at long times to a deterministic value related to certain ensemble averages of the full problem.  Such results are important 
in justifying the utility of studying the ensemble:  an experimentalist only needs to perform an experiment with one single realization of the random flow for the ensemble averages to make physically relevant predictions.  
The tools for these results are based on the conclusion of the Hausdorff moment problem \cite{shohat1943problem}, the ground state eigenvalue perturbation approach proposed in  \cite{bronski1997scalar}, the availability of closed moment equations for flows involving OU processes \cite{resnick1996dynamical}, and the available exact PDF for white wind models \cite{bronski2007explicit}, here extended to OU wind models.  Figure \ref{fig:Overview} gives a schematic diagramming the theoretical approach we take connecting the effective equations to the original problem through Hausdorff, and the results which ultimately follow from this connection.

\begin{figure}
  \centering
    \includegraphics[width=1\linewidth]{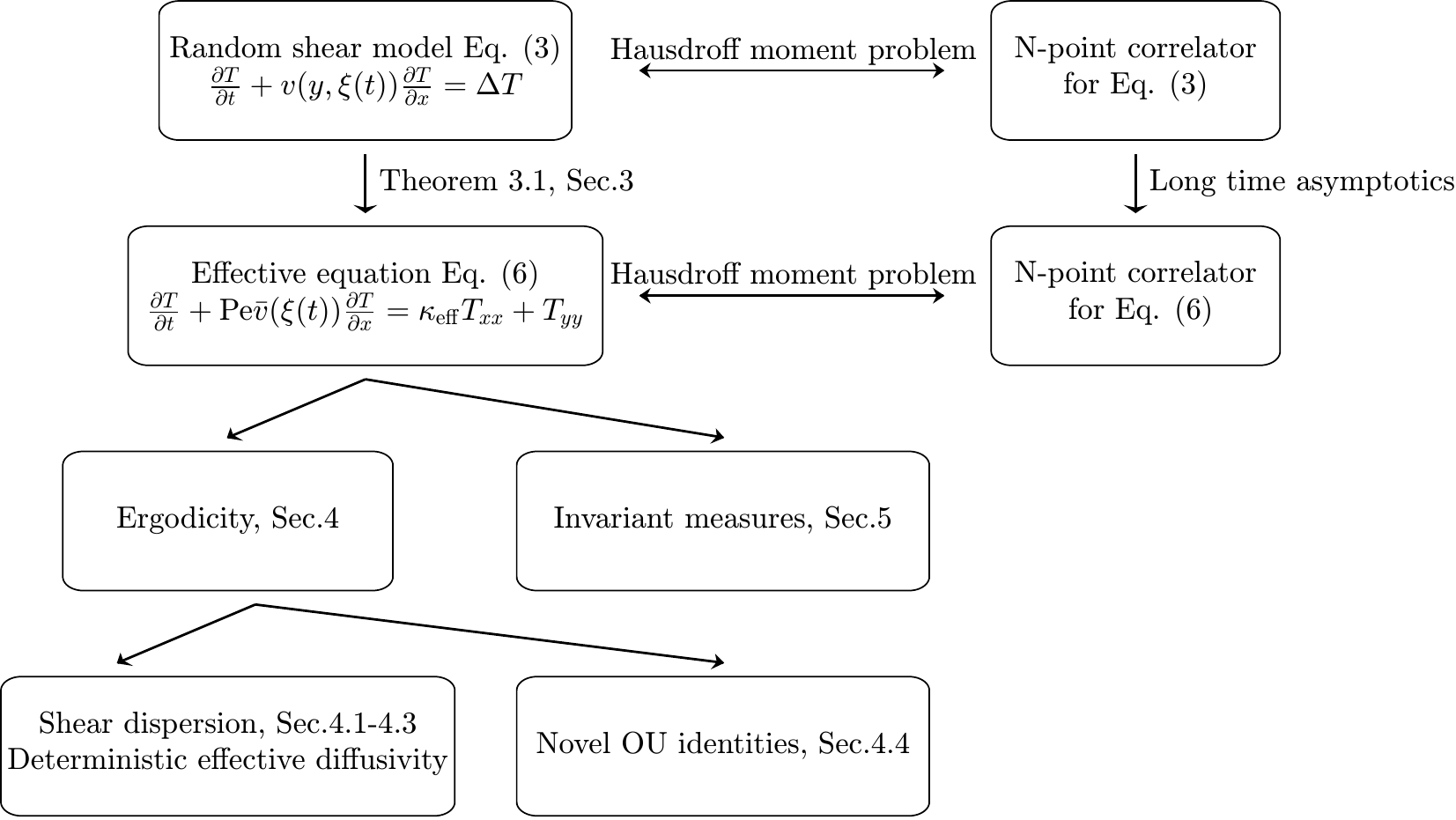}
  \caption{Main results and the structure of their derivation. }
  \label{fig:Overview}
\end{figure}

The paper is organized as follows, as well as summarized in figure \ref{fig:Overview}: In section \ref{sec:setup}, we formulate the evolution of the passive scalar field advected by a random shear flow which depends on a stationary OU process in a parallel-plate channel, which is a generalization of Majda model. Then we review the some important conclusions of the associated moment closure problem. In section \ref{sec:MainResult}, we derive the effective advection-diffusion equation, which is the key result of this paper. In section \ref{sec:ShearDispersion}, we show the link between the Kraichnan-Majda model and Taylor-Aris Dispersion. In section \ref{sec:scalarIntermittency}, we derive the long time invariant measure for three classes of initial data:1) deterministic initial data, 2) square integrable Gaussian random initial data, 3) wave function data with a Gaussian random amplitude. In section \ref{sec:discuss}, we summarize the conclusions from the findings in the paper and briefly discuss future studies.


\section{Setup and background of the Majda Model}
\label{sec:setup}
In this paper, we will study the following random advection diffusion equation with initial condition $T_{0}\left(x,y\right)$ and impermeable channel
boundary conditions,
\begin{equation}\label{eq:advectionDiffusion}
\frac{\partial T}{\partial t}+v(y,\xi(t))\frac{\partial T}{\partial x}=\kappa \Delta T \,,\qquad
 T(x,y,0)=T_{0}(x,y)\,, \qquad
\displaystyle\left. \frac{\partial T}{\partial y}\right|_{y= 0,L}=0 \, , 
\end{equation}
where the domain is $\left\{ (x,y)| x\in \mathbb{R}, y \in  [0, L] \right\}$, $L$ is the gap thickness of the channel, $\kappa$ is the diffusivity, $\xi(t)$ is a  zero-mean, Gaussian random process with the correlation function given by $\left\langle \xi(t)\xi(s) \right\rangle=R(t,s)$. A special case of flow $v(y,\xi(t))$ is the multiplicatively separable function $v (y,\xi(t))=u(y)\xi(t)$. This type of flow can  originate from either a time varying pressure field, or by randomly moving portions of the boundary, in a high viscosity fluid, see section 2 of \cite{camassa2019symmetry} for more details. Two types of $\xi (t)$ received great attention in the literature:  1) $\xi(t)$ is a Gaussian white noise in time so that $R(t,s)=g^{2} \delta(t-s)$, or 2) $\xi(t)$ is a stationary Ornstein-Uhlenbeck process with damping $\gamma$ and dispersion $\sigma$, which is the solution of stochastic differential equation (SDE) $\mathrm{d}\xi (t) =-\gamma \xi (t)\mathrm{d}t +\sigma \mathrm{d}B (t)$ with initial condition $\xi(0) \sim \mathcal{N} (0, {\sigma^2}/{2 \gamma})$. Here $B (t)$ is the standard Brownian motion and $\mathcal{N} (a,b)$ is the normal distribution with mean $a$ and variance $b$.   The correlation function of $\xi (t)$ is $R(t,s)=\frac{\sigma^{2}}{2\gamma}e^{-\gamma \left| t-s \right|}$. $\gamma^{-1}$ is often referred to as the correlation time of the OU process. It is easy to check that the stationary Ornstein-Uhlenbeck process converges to the Gaussian white noise process as the correlation time vanishes with fixed ${\sigma}/{\gamma}$. Due to this property, we will focus on the OU process and consider the white noice process as a limiting case in this paper. 

Notice that $\gamma \sim \text{Time}^{-1}$, $\sigma \sim \text{Time}^{ -\frac{3}{2}}$. With the change of variables, 
\begin{equation}
\begin{array}{ccc}
Lx'= x & Ly'=y&\frac{L^2}{\kappa}t'=t \\ 
g=\frac{\sigma}{\gamma}&  \frac{\kappa}{L^2}\gamma'=\gamma  &U= Lg^{2}  \\
Uv'(y,\xi' (t))=v (y, \xi (t)) &\frac{g \sqrt{\kappa}}{L}\xi'(\frac{L^2}{\kappa}t')=\xi (t)  &T' (Lx',Ly',\frac{L^2}{\kappa}t')=T (x,y,t), \\
\end{array}
\end{equation}
we can drop the primes without confusion and obtain the nondimensionalized version of \eqref{eq:advectionDiffusion}:
\begin{equation}\label{eq:advectionDiffusionNonDimension}
\displaystyle \frac{\partial T}{\partial t}+\text{Pe} v(y,\xi(t))\frac{\partial T}{\partial x}=\Delta T\,,\qquad
\displaystyle T(x,y,0)=T_{0}(x,y)\,,\qquad
\displaystyle \left.\frac{\partial T}{\partial y}\right|_{y= 0,1}=0\,,
\end{equation}
where the domain is $\left\{ (x,y)| x\in \mathbb{R}, y \in  [0,1] \right\}$, and we have introduced the P\'{e}clet number  $\text{Pe}=  {U L}/{ \kappa}=  {L^{2}g^{2}}/\kappa$. When $\xi (t)$ is the white noise process, the correlation function of $\xi (t)$ is $R(t,s)= \delta (t-s)$.  Conversely, when $\xi (t)$ is the stationary Ornstein-Uhlenbeck process, the underlying SDE becomes $\mathrm{d}\xi (t) =-\gamma \xi (t)\mathrm{d}t + \mathrm{d}B (t)$ with the initial condition $\xi (0) \sim \mathcal{N} (0,\frac{\gamma}{2})$, and the correlation function of $\xi (t)$ is $R(t,s)= \frac{\gamma}{2}e^{-\gamma \left| t-s \right|}$.

Define the $N$-point correlation function $\mathbf{\Psi}_{N}$ of the scalar field $T(x,y,t)$: $\mathbb{R}^{N}\times \mathbb{R}^{N}\times \mathbb{R}^{+}\rightarrow \mathbb{R}$ by $\mathbf{\Psi}_N (\mathbf{x}, \mathbf{y},t) =\left<\prod_{j=1}^N T(x_j,y_j,t)\right >_{\xi (t)}$, where $\mathbf{x}=\left(x_1,x_2,\cdots,x_N\right)$, $\mathbf{y}=(y_1,y_2,\cdots,y_N)$. Here, the brackets $\left\langle \cdot  \right\rangle_{\xi(t)}$ denote ensemble averaging with respect to the stochastic process $\xi(t)$.  The $\mathbf{\Psi}_{N}$ associated with the free space version of \eqref{eq:advectionDiffusion} is known for some special flows. When $ v(y,\xi(t))=u (y)\xi (t)$ and $\xi(t)$ is  the Gaussian white noise process, Majda \cite{majda1993random} showed that $\hat{\mathbf{\Psi}}_{N}$ satisfies a $N$-body parabolic quantum mechanics problem,
\begin{equation}\label{closureeqnWhite}
\begin{aligned}
\frac{\partial \hat{\mathbf{\Psi}}_{N}}{\partial t} &= \Delta_N \hat{\mathbf{\Psi}}_{N}- \left( \frac{\text{Pe}^2}{2} \left( \sum\limits_{j=1}^{N}u\left(y_{j}\right) k_{j}\right)^2+ |\mathbf{k}|^2  \right)\hat{\mathbf{\Psi}}_{N}\\
 \hat{\mathbf{\Psi}}_{N}(\mathbf{k},\mathbf{y},0)&= \prod_{j=1}^N \hat{T}_{0}(k_j,y_j)
\end{aligned}
\end{equation}
where $\hat{f}(\mathbf{k})= \int\limits_{\mathbb{R}^N}^{} \mathrm{d}\mathbf{y} e^{\mathrm{i}(\mathbf{x}\cdot \mathbf{k} )}f(\mathbf{x})$ is the Fourier transformation of $f(\mathbf{x})$, $\Delta_N$ is the Laplacian operator in $N$ dimensions $\Delta_{N}=\sum\limits_{j=1}^N \frac{\partial^2}{\partial y_j^{2}}$,  $\mathbf{k}=\left(k_1,k_2,\cdots,k_N\right)$. When $u (y)=y$, Majda \cite{majda1993random}  derived the exact expression of $\Psi_N$.  A rotation of coordinates reduces the $N$-dimensional problem to a one-dimensional problem. Then the solution of \eqref{closureeqnWhite} is available via Mehler's formula.
Based on this exact $N$-point correlation function, the distribution of the scalar field advected by a linear shear flow has been studied for deterministic and random initial data. The non-Gaussian behaviors of PDF have been reported in  \cite{mclaughlin1996explicit,bronski2000problem,bronski2000rigorous}.

When $\xi(t)$ is the stationary Ornstein-Uhlenbeck process, by introducing an extra variable $z$, Resnick \cite{resnick1996dynamical} showed  that $\hat{\mathbf{\Psi}}_{N} (\mathbf{k}, \mathbf{y},t)= \frac{1}{\sqrt{\pi}}\int\limits_{-\infty}^{+\infty} \hat{\psi} (\mathbf{k}, \mathbf{y},z,t) e^{-z^2} \mathrm{d} z $, where $\hat{\psi} (\mathbf{k}, \mathbf{y},z,t) $ satisfies the following partial differential equation:
\begin{equation}\label{closureeqnOU}
\begin{aligned}
\frac{\partial \hat{\psi}}{\partial t}+\mathrm{i} \text{Pe}\sum\limits_{j=1}^N   v(y_{j}, \sqrt{\gamma} z)  k_{j}  \hat{\psi}+ \gamma z \hat{\psi}_z &= \Delta_N \hat{\psi}- |\mathbf{k}|^{2}\hat{\psi}+ \frac{\gamma}{2}\hat{\psi}_{zz}\\
 \hat{\psi}(\mathbf{k},\mathbf{y},z,0)&=\prod_{j=1}^N \hat{T}_{0}(k_j,y_j)
\end{aligned}
\end{equation}
When $u (y)=y$, Resnick \cite{resnick1996dynamical} derived the exact expression for $\Psi_{N}$ via the similar strategy Majda used for solving \eqref{closureeqnWhite} and showed it converges to the solution of \eqref{closureeqnWhite} in the limit  $\gamma\rightarrow \infty$.

These results are all derived in free-space.   The analytic formula of the $N$-point correlation function $\Psi_{N}$ for the boundary value problem \eqref{eq:advectionDiffusionNonDimension} is unknown even for simple-geometry domains.  For periodic boundary conditions, Bronski and McLaughlin \cite{bronski1997scalar} carried out a second order perturbation expansion for the ground state of periodic Schr\"odinger equations to analyze the inherited probability measure for a passive scalar field advected by periodic shear flows with multiplicative white noise.
In  \cite{camassa2019symmetry,camassa2020persisting}, equation \eqref{eq:advectionDiffusionNonDimension} was studied with a stationary OU process, where a dramatically different long time state resulting from the existence of the impermeable boundaries was found. In particular, the PDF of the scalar in the channel case has negative skewness for sufficiently small P\'{e}clet number,  in stark contrast to free space, where the limiting skewness is strictly positive for all P\'{e}clet number. Inspired by the observation, we further explore here the PDF of the advected scalar in the presence of impermeable boundaries by the perturbation method introduced in \cite{bronski1997scalar}.  Briefly, the long time behavior of the Fourier transformation of $N$-point correlation function $\hat{\Psi}_{N}$ of the scalar field is dominated by the neighborhood of the zero frequency  $\mathbf{k}=\mathbf{0}$. This observation reduces the series expansion of $\hat{\Psi}_{N}$ to a single multi-dimensional Laplace type integral. Then, the standard asymptotic analysis and inverse Fourier transformation yield the long time asymptotic expansion of $\Psi_N$.

\section{Effective Equation at long times}
\label{sec:MainResult}
We begin by stating the key result of the paper as a theorem. In the following context, we use $\bar{a}$ to denote the cross sectional average of function $a (y)$, $\bar{a}=\int\limits_0^1 a(y)\mathrm{d} y$.
\begin{theorem}\label{thm:WindModelApproximation}
Assume $v (y, \sqrt{\gamma}z)$ has the Hermite polynomial series representations $v (y,\sqrt{\gamma}z)= \sum\limits_{n=0}^{\infty}a_n (y,\sqrt{\gamma}) H_n (z)$, where $\bar{a}_{0}=0$ and $H_n (z)$ is the $n$-th Hermite polynomial which is the orthogonal polynomial with respect to the weight function $e^{-z^2}$.  The solution of equation \eqref{eq:advectionDiffusionNonDimension}can be approximated by the solution of the following equation (wind model) at long times:
\begin{equation}\label{eq:WindModel}
\begin{aligned}
 \frac{\partial T}{\partial t}+  \mathrm{Pe}\bar{v}(\xi(t))\frac{\partial T}{\partial x}=\kappa_{\mathrm{eff}} T_{xx}+  T_{yy}, \quad 
 T(x,y, 0)= T_0(x,y),\quad \left. \frac{\partial T}{\partial y} \right|_{y= 0,1}=0\\
\end{aligned}
\end{equation}
 where  $\bar{v} (z)= \int\limits_{0}^{1} v(y,z)\mathrm{d} y $, $\kappa_{\mathrm{eff}}= \frac{\lambda^{(2)}-\lambda^{(1,1)}}{2}$ and 
\begin{equation}
\begin{aligned}
&\lambda_0^{(2)}= 2+2\mathrm{Pe}^{2}\sum\limits_{n=0}^{\infty} n! 2^{n}\int\limits_{0}^{1}a_n (y)\left( n\gamma-\Delta \right)^{-1}a_{n} (y)\mathrm{d} y \\
&\lambda_0^{(1,1)}=\frac{2\mathrm{Pe}^{2}}{\gamma}\sum\limits_{n=1}^{\infty}  (n-1)! 2^{n}\bar{a}_n^{2} 
=\frac{4\mathrm{Pe}^{2} }{\gamma} \int\limits_{-\infty}^{\infty} e^{z^2} \left( \int\limits_{-\infty}^z e^{-z^2}\bar{v} (z_{1})\mathrm{d}z_1 \right)^{2}\mathrm{d}z\\
&\left( \lambda-\Delta \right)^{-1}a (y) = \frac{1}{\sqrt{\lambda }}\left( \frac{\cosh \left(\sqrt{\lambda } y\right) \int_0^1 a(s) \cosh \left(\sqrt{\lambda } (1-s)\right) \mathrm{d} s}{\sinh \left(\sqrt{\lambda } \right)}\right.\\
&\hspace{1cm} \left. -\int_0^y a(s) \sinh \left(\sqrt{\lambda } (y-s)\right) \mathrm{d} s \right) \\
&\Delta^{-1}a (y) =-\int\limits_{0}^y \int\limits_{0}^{y_1}a(y_{2})\mathrm{d}y_2\mathrm{d}y_1,\; \text{if}\;  \bar{a}=0\\
\end{aligned}
\end{equation}

\end{theorem}

\begin{proof}
To show two random field are same at long times, we need to prove they have the same $N$-point joint distribution at long times. Due to the maximum principle of heat equation, the solution is bounded by the maximum value of initial condition. The Hausdorff moment problem \cite{shohat1943problem} concludes that, for the random variable supported on a closed interval, the sequence of moments are uniquely determine the distribution. Hence, we only need to show that the solution of equation \eqref{eq:advectionDiffusionNonDimension} and equation \eqref{eq:WindModel} have the same $N$-point correlation function at long times.  

We derive the long time asymptotic expansion of $N$-point correlation function of equation \eqref{eq:advectionDiffusionNonDimension} with exponential decay correction with the ground state energy expansion strategy described in \cite{bronski1997scalar,camassa2020persisting}. 
\begin{equation}\label{eq:eigenfunctionExpansion0th}
\begin{aligned}
\mathbf{\Psi}_{N}(\mathbf{x}, \mathbf{y},t)= \frac{1}{(2\pi)^{N}} \int\limits_{\mathbb{R}^{N}}^{}  e^{-\mathrm{i}(\mathbf{x}\cdot \mathbf{k})}\beta_0(\mathbf{k}) \phi_0 (\mathbf{k},\mathbf{y})
          e^{-\lambda_0 (\mathbf{k}) t}\mathrm{d}\mathbf{k} +\mathcal{O} (e^{-\pi^{2}t}) \text{ as } t\rightarrow \infty
\end{aligned}
\end{equation}
where $\lambda_{0} (\mathbf{k}), \phi_0 (\mathbf{k},\mathbf{y})$ are the first eigenvalues and eigenfunctions ($l=0$) of the eigenvalue problem:
\begin{equation}\label{eq:eigenvalueProblemOU}
\begin{aligned}
-(\lambda_{l}-\left| \mathbf{k} \right|^{2}) \varphi_{l}&=-\mathrm{i} \mathrm{Pe} \sum\limits_{j=1}^N k_{i} v(y_{j}, \sqrt{\gamma} z)  \varphi_{l}- \gamma z \frac{\partial \varphi_{l}}{\partial z}+ \frac{\gamma}{2} \frac{\partial^2 \varphi_l}{\partial z^{2}}  + \Delta_N \varphi_{l}\\
\frac{\partial \varphi_{l}}{\partial y_{j}}|_{y_{j}= 0, 1}&=0 \quad \forall  1\leq j\leq N
\end{aligned}
\end{equation}
Here, we choose $\varphi_{l}$ such that $\left\{ \varphi_{l} \right\}_{l=0}^{\infty}$ form an orthonormal basis with respect to the inner product $ \left\langle f(\mathbf{y},z),g(\mathbf{y},z) \right\rangle=\frac{1}{\sqrt{\pi}} \int\limits_{-\infty}^{+\infty}\mathrm{d}z \int\limits_{[0,1]^{N}}^{} f(\mathbf{y},z)g^{*}(\mathbf{y},z)e^{-z^{2}}\mathrm{d} \mathbf{y}$ respectively, where $g^{*}$ is the complex conjugate of $g$. With this definition of inner product, we have  $\beta_l(\mathbf{k})= \left\langle \prod_{j=1}^N \hat{T}_{0}(k_j,y_j),  \phi_l(\mathbf{k},\mathbf{y}) \right\rangle  $.

Equation \eqref{eq:eigenfunctionExpansion0th} is a $N$ dimensional Laplace type integral with respect to the frequency variable $\mathbf{k}$. It is well known that, for $t$ sufficiently large, the integral becomes localized near the minimum of $\lambda(\mathbf{k})$\cite{inglot2014simple,kirwin2010higher}. Applying the regular perturbation theory on the eigenvalue problem \eqref{eq:eigenvalueProblemOU} yields that $\lambda_0 (\mathbf{0})=0, \phi_{0}(\mathbf{0},\mathbf{y})=1$ ( see details in appendix \ref{sec:PerturbationImpermeable}). Hence, we have the approximation of  \eqref{eq:eigenfunctionExpansion0th} as $t\rightarrow \infty$:
\begin{equation}\label{eq:NpointCorrelationShear}
\begin{aligned}
\mathbf{\Psi}_{N}(\mathbf{x}, \mathbf{y},t)&= \frac{1}{(2\pi)^{N}} \int\limits_{\mathbb{R}^{N}}^{}  e^{-\mathrm{i}(\mathbf{x}\cdot \mathbf{k})}\beta_0(\mathbf{k}) \phi_0 (\mathbf{0},\mathbf{y})
          e^{- \frac{1}{2}\mathbf{k} \mathbf{\Lambda}_{1}\mathbf{k}^{\mathrm{T}} t}\mathrm{d}\mathbf{k} +\mathcal{O} (t^{-\frac{N+2}{2}} )  \\
&=  \frac{1}{(2\pi)^{N}} \int\limits_{\mathbb{R}^{N}}^{} \left(\int\limits_{[0,1]^N}^{}\hat{\mathbf{\Psi}}_{N}(\mathbf{k},\mathbf{y},0) \mathrm{d} \mathbf{y} \right) 
e^{-\mathrm{i}(\mathbf{x}\cdot \mathbf{k})- \frac{1}{2}\mathbf{k} \mathbf{\Lambda}_{1}\mathbf{k}^{\mathrm{T}} t}\mathrm{d}\mathbf{k} +\mathcal{O} (t^{-\frac{N+2}{2}} ) \\
&=  \frac{\exp \left( - \frac{1}{2t}\mathbf{x} \mathbf{\Lambda}_{1}^{-1}\mathbf{x}^{\mathrm{T}} \right) }{(2\pi)^{\frac{N}{2}} \sqrt{ \det (\Lambda_{1})}}\int\limits_{[0,1]^N}^{}\hat{\mathbf{\Psi}}_{N}(\mathbf{0},\mathbf{y},0) \mathrm{d} \mathbf{y}  
           +\mathcal{O} (t^{-\frac{N+2}{2}} ) \\
\end{aligned}
\end{equation}
where $\left( \mathbf{\Lambda}_{1} \right)_{i,j}= \frac{\partial^2 }{\partial k_i \partial k_j} \lambda_0(\mathbf{x},\mathbf{k})|_{\mathbf{k}=\mathbf{0}}$ is the Hessian matrix of the eigenvalue $\lambda_0 (\mathbf{k})$ at $\mathbf{k}=\mathbf{0}$. Since the eigenvalue problem  \eqref{eq:eigenvalueProblemOU} are invariant under the permutation of frequency variables, $\mathbf{\Lambda}_{1}$  only depends on the derivative of eigenvalue in one-dimensional eigenvalue problem $\lambda^{(2)}=\frac{\partial^{2}}{\partial k_1^{2}}\lambda_0 (k_{1})|_{k_1=0} $ and the derivative of eigenvalue in two-dimensional eigenvalue problem $\lambda^{(1,1)}=\frac{\partial^{2}}{\partial k_1\partial k_2}\lambda_0 (k_{1},k_2)|_{k_1=0,k_2=0} $.
Therefore we have that $\mathbf{\Lambda}_{1}= \left( \lambda^{(2)} -\lambda^{(1,1)} \right)\mathbf{I}+\lambda^{(1,1)}\mathbf{e}^{\text{T}}\mathbf{e}$, where $\mathbf{I}$ is the identity matrix of size $N\times N$ and $\mathbf{e}$ is a $1\times N$ vector with $1$ in all coordinates. The explicit formula of  $\lambda^{(2)},\lambda^{(1,1)}$ can be obtained by the  perturbation method introduced in the appendix of \cite{bronski1997scalar}. Appendix \ref{sec:PerturbationImpermeable} shows the details of the calculation. We should remark that series formula of the $\lambda^{(2)},\lambda^{(1,1)}$ we presented here may not be optimally convergent. One can choose different basis to solve the recursive system based on the form of $v(y,z)$.

The same strategy leads to the $N$-point correlation of the solution of the  equation \eqref{eq:WindModel} (also see section 4 of \cite{camassa2020persisting} for details),
\begin{equation}\label{eq:NpointCorrelationWind}
\begin{aligned}
\mathbf{\Psi}_{N}(\mathbf{x}, \mathbf{y},t)&=  \frac{1}{(2\pi)^{N}} \int\limits_{\mathbb{R}^{N}}^{} \left(\int\limits_{[0,1]^N}^{}\hat{\mathbf{\Psi}}_{N}(\mathbf{k},\mathbf{y},0) \mathrm{d} \mathbf{y} \right)  e^{-\mathrm{i}(\mathbf{x}\cdot \mathbf{k})} 
          e^{- \mathbf{k} \mathbf{\Lambda}_2\mathbf{k}^{\mathrm{T}} t} \mathrm{d} \mathbf{k}\\
\end{aligned}
\end{equation}
where $\mathbf{\Lambda}_{2}= 2\kappa_{\mathrm{eff}}\mathbf{I}+\lambda^{(1,1)}\mathbf{e}^{\text{T}}\mathbf{e}$. We can see that the $N$-point correlation function \eqref{eq:NpointCorrelationShear} and \eqref{eq:NpointCorrelationWind} are the same when $\kappa_{\mathrm{eff}}= \frac{\lambda^{(2)}-\lambda^{(1,1)}}{2}$. Therefore, the solutions of equation \eqref{eq:advectionDiffusionNonDimension} and equation \eqref{eq:WindModel} have the same $N$-point correlation function at long times. This completes the proof.

\end{proof}

\begin{remark}
Theorem \ref{thm:WindModelApproximation} holds for the periodic boundary condition  except with a different definition of operator $\left( \lambda-\Delta \right)^{-1}$.
\begin{equation}
\begin{aligned}
\left( \lambda-\Delta \right)^{-1} a (y)= &\frac{ \sinh \left(\sqrt{\lambda } \left(y-\frac{1}{2}\right)\right)\int_0^1 a(s) \sinh \left(\sqrt{\lambda } (L-s)\right) \, ds}{2 \sqrt{\lambda } \text{sinh}\left(\frac{\sqrt{\lambda }}{2}\right) }\\
&+\frac{\cosh \left(\sqrt{\lambda } \left(y-\frac{1}{2}\right)\right) \int_0^1 a(s) \cosh \left(\sqrt{\lambda } (1-s)\right) \, ds}{2 \sqrt{\lambda } \text{sinh}\left(\frac{\sqrt{\lambda }}{2}\right) }\\
& -\frac{\int_0^y a(s) \sinh \left(\sqrt{\lambda } (y-s)\right) \, ds}{\sqrt{\lambda }} \\
\Delta^{-1}a (y) =&- \int\limits_{0}^y \int\limits_{0}^{y_1} a (y_2)\mathrm{d}y_{2}\mathrm{d} y_1 + y\int\limits_{0}^{1}a (y_2)\mathrm{d}y_2+ \int\limits_{0}^{1} \int\limits_{0}^{y_1} a (y_2)\mathrm{d}y_{2}\mathrm{d} y_1 - \int\limits_{0}^{1}a (y_2)\mathrm{d}y_2\\
\end{aligned}
\end{equation}
Appendix \ref{sec:PerturbationPeriodic} shows the details of the calculation. 

\end{remark}

\begin{remark}
The condition $\bar{a}_{0} =0$ is introduced for the convenience of analysis. For functions $a_{0} (y)$ which do not satisfy this condition, one can apply the Galilean transformation $\tilde{x}=x- t\int\limits_{0}^{1}a_{0} (y)\mathrm{d} y$ so that $T (\tilde{x},y,t)$ satisfies a same equation with zero cross sectional average function  $\tilde{a}_0 (y)=a_{0} (y)-\int\limits_{0}^{1}a_{0} (y)\mathrm{d} y$.

\end{remark}

\begin{remark}
A special case of flow $v(y,\xi(t))$ is the multiplicatively separable function $v (y,\xi(t))=u(y)\xi(t)$, which has received considerable interest in the literature \cite{majda1993random, mclaughlin1996explicit, bronski1997scalar}. In this case, $v(y,\sqrt{\gamma}z)$ have the Hermite polynomial expansion with coefficients $a_{1}= \frac{\sqrt{\gamma}}{2}u(y) $, $a_{n}=0$ if $n=0 ,\text{or } n \geq 2$. By the theorem \ref{thm:WindModelApproximation}, we have  $\kappa_{\mathrm{eff}}= \frac{\lambda^{(2)}-\lambda^{(1,1)}}{2}$ and 
  \begin{scriptsize}
\begin{equation}\label{eq:eigenvalueOU}
\begin{aligned}
  & \lambda^{(2)}= 2+\mathrm{Pe}^{2} \sqrt{\gamma } \int_{0}^{1}u(y) \left(\frac{\cosh \left(\sqrt{\gamma } \left(y+\frac{1}{2}\right)\right)}{\sinh\left(\sqrt{\gamma }\right)}  \int_{0}^{1}\mathrm{d}s u(s) \cosh \left(\sqrt{\gamma } \left(\frac{1}{2}-s\right)\right) -\int_{0}^y \mathrm{d}su(s) \sinh \left(\sqrt{\gamma } (y-s)\right) \right) \mathrm{d}y  \\
  &\lambda^{(1,1)}=\mathrm{Pe}^{2}\left( \int\limits_{0}^{1} u(y) \mathrm{d} y\right)^{2} \\
\end{aligned}
\end{equation}
\end{scriptsize}

When $\gamma\rightarrow \infty$, the stationary Ornstein-Uhlenbeck Process converges to the Gaussian white noise process  and $\lambda^{(2)},\lambda^{(1,1)}$ converge to
\begin{equation}\label{eq:eigenvalueWhite}
\begin{aligned}
 \lambda^{(2)}=2+ \mathrm{Pe}^2\int\limits_{0}^{1} u^2(y)\mathrm{d}y,&\quad  \lambda^{(1,1)}= \mathrm{Pe}^{2}\left( \int\limits_{0}^{1} u(y) \mathrm{d} y\right)^{2}\\
\end{aligned}
\end{equation}
\end{remark}

In the following sections, we will elaborate the application of theorem \ref{thm:WindModelApproximation} in the field of shear dispersion and scalar intermittency.
\section{Shear Dispersion and Ergodicity}
\label{sec:ShearDispersion}
The theorem \ref{thm:WindModelApproximation} surprisingly shows that the stationary Ornstein-Uhlenbeck process dependent random shear flow induces a deterministic effective diffusivity at long times. In contrast, in free space, the effective diffusivity (the normalized, centered, second spatial moment of the scalar) is random and time dependent, see detailed discussion in section \ref{sec:DeterministicInitialData}. In this section, with the effective advection-diffusion equation derived in the theorem \ref{thm:WindModelApproximation}, we will show the connection to the Taylor dispersion, the ergodicity of the random field and the long time asymptotic expansion of  OU process related time integral.

\subsection{Taylor Dispersion}
The theorem \ref{thm:WindModelApproximation} also provides a stochastic proof for Taylor dispersion induced by steady flow. We can eliminate the stationary Ornstein-Uhlenbeck process dependence of flow and obtain the steady flow either by choosing $v(y,z)=v(y)$ or letting the dispersion parameter $\sigma$ of stationary OU process to be zero in the dimensional equation \eqref{eq:advectionDiffusion}. The effective diffusivity in this case is 
\begin{equation}
\begin{aligned}
\kappa_{\mathrm{eff}}= &1-\frac{1}{2}\mathrm{Pe}^{2} \int\limits_{0}^{1}v (y)\int\limits_{0}^y \int\limits_{0}^{y_1}v(y_{2})\mathrm{d}y_2\mathrm{d}y_1 \mathrm{d} y 
 \\
=  &1+\frac{1}{2}\mathrm{Pe}^{2} \int\limits_{0}^{1} \left( \int\limits_{0}^y v(y_{1})\mathrm{d}y_1 \right)^{2} \mathrm{d} y  \\
\end{aligned}
\end{equation} 
where the second step follows the integration by parts. This is the formula of the Taylor dispersion induced by a steady shear flow in a  parallel-plate channel in \cite{camassa2010exact,mercer1990centre}.

The enhanced diffusivity induced by the steady flow is inversely proportional to the molecular diffusivity \cite{taylor1953dispersion,camassa2010exact}, the one induced by periodic time varying flow is proportional to the molecular diffusivity \cite{ding2020enhanced,jimenez1984contaminant,chatwin1975longitudinal}. However, for the random flow we study in this paper, the enhanced diffusivity behaves differently. For example, when $v(y,\xi (t))= y\xi(t)$, the effective diffusivity in dimensional form is 
\begin{equation}\label{eq:effDiffusivityExample1}
\begin{aligned}
\kappa_{\mathrm{eff}}= &\kappa+g^2 \left( \frac{ L^2}{24}-\frac{ \kappa }{2 \gamma }+\frac{ \kappa ^{3/2} \tanh \left(\frac{\sqrt{\gamma } L}{2 \sqrt{\kappa }}\right)}{\gamma ^{3/2} L} \right)\\
\end{aligned}
\end{equation} 
There is a term in the expression for the effective diffusivity which is independent on $\kappa$ and a term which is nonlinearly dependent on $\kappa$. An extreme case is the zero correlation time $\gamma^{-1}=0$, where $\xi(t)$ becomes the Gaussian white noise process and equation \eqref{eq:effDiffusivityExample1} becomes 
\begin{equation}\label{eq:effDiffusivityExample2}
\begin{aligned}
\kappa_{\mathrm{eff}}= &\kappa+ \frac{ L^2g^{2}}{24}\\
\end{aligned}
\end{equation} 
The equation \eqref{eq:effDiffusivityExample2} implies that the enhanced diffusivity is totally independent on $\kappa$ when the correlation time vanishes. For the other limit of the correlation time $\gamma^{-1} \rightarrow \infty$, the effective diffusivity has the following asymptotic expansion
\begin{equation}
\begin{aligned}
\kappa_{\mathrm{eff}}= &\kappa+\frac{\gamma  g^2 L^4}{240 \kappa }+\mathcal{O} (\gamma^{\frac{3}{2}})\\
\end{aligned}
\end{equation}
When the correlation time is longer, the flow behaves more like the deterministic steady flow.

\begin{figure}
  \centering
    \includegraphics[width=5.5cm, height=3.4cm]{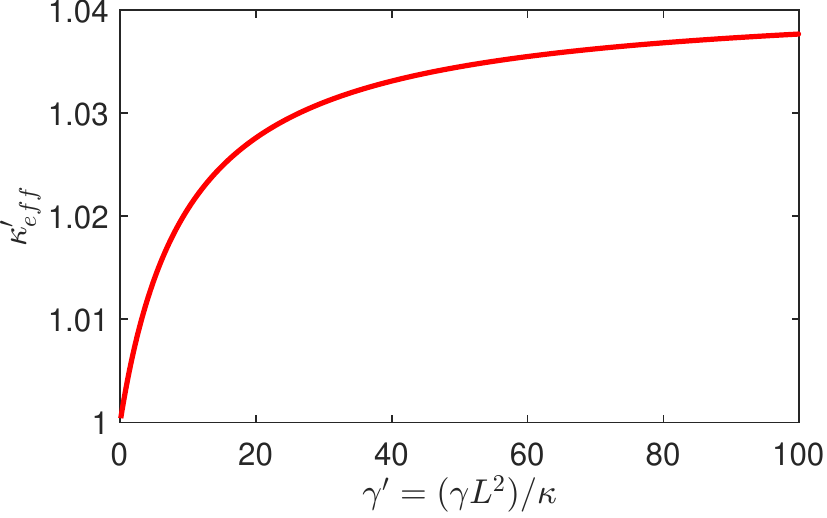}
  \caption{The non-dimensional effective diffusivity $\kappa'_{\mathrm{eff}}$for various non-dimensional damping parameter $\gamma'=\gamma L^{2}/\kappa $ }
  \label{fig:KappaEffgamma}
\end{figure}

 

We can interpret the eigenvalue as the energy of the associated two particle system similar to what was done by Bronski and McLaughlin \cite{bronski1997scalar}. Since the ground state energy of the fully interacting two particle problem is lower than the energy of two independent particle problems, we always have $\lambda^{(2)}-2 \geq \lambda^{(1,1)}$. The difference yields the enhanced diffusion. Hence, similar to the deterministic flow, the enhanced diffusivity vanishes if and only if there is no spatial dependence in the flow, that is, $v(y,z)=v(z)$.

\subsection{Zero diffusivity}
As the molecular diffusivity tends to zero, the non-dimensional $\gamma'={L^{2}\gamma}/{\kappa} \rightarrow \infty$, where $\xi (t)$ converges to the Gaussian white noise.
However, when the molecular diffusivity is exactly zero, the effective diffusivity would be random. To further understand this, we consider the equation \eqref{eq:advectionDiffusionNonDimension} with line source initial data $T_0(x,y)=\delta (x)$ and without the diffusion term. In this case, the equation can be solved by the method of characteristic:
\begin{equation}
\begin{aligned}
T (x,y,t)= & \delta (x- u (y)\int\limits_0^t\xi (s)\mathrm{d}s)\\
\end{aligned}
\end{equation}
Then it leads to the first and second Aris moment
\begin{equation}
\begin{aligned}
\bar{T}_1= &\int\limits_0^t\xi (s)\mathrm{d}s  \int\limits_0^1 u (y)\mathrm{d} y \\
\bar{T}_2= &  \left(  \int\limits_0^t\xi (s)\mathrm{d}s \right)^2\int\limits_0^{1} u^{2} (y) \mathrm{d} y\\
\end{aligned}
\end{equation}
Base on the formula \eqref{eq:ArisEffectiveDiffusion}, we have
\begin{equation}
\begin{aligned}
\kappa_{\mathrm{eff}}= & \left(\int\limits_0^{1} u^{2} (y) \mathrm{d} y - \left( \int\limits_0^{1} u (y) \mathrm{d} y \right)^{2}  \right)    \lim\limits_{t\rightarrow \infty} \frac{ 1}{2t}\left(  \int\limits_0^t\xi (s)\mathrm{d}s \right)^2\\
=&\left(\int\limits_0^{1} u^{2} (y) \mathrm{d} y - \left( \int\limits_0^{1} u (y) \mathrm{d} y \right)^{2}  \right)  \frac{1}{2} B^{2} (1)\\
\end{aligned} 
\end{equation}
In this case, $\kappa_{\mathrm{eff}}$ is a random variable. However, if we consider its assemble average with respect to the stochastic process $\xi (t)$, we have
\begin{equation}\label{eq:ArisEffectiveDiffusionAssemble}
\begin{aligned}
\left\langle \kappa_{\mathrm{eff}} \right\rangle_{\xi (t)} =&\frac{1}{2}\left(\int\limits_0^{1} u^{2} (y) \mathrm{d} y - \left( \int\limits_0^{1} u (y) \mathrm{d} y \right)^{2}  \right)   \\
\end{aligned}
\end{equation}
The above equation shows that, when molecular diffusivity becomes zero, the equation \eqref{eq:effDiffusivityExample2} is still valid in the sense of assemble average.

\subsection{Ergodicity}
In this section, we will show the ergodicity of the OU process yields the ergodicity of the random passive scalar field. More precisely, we can construct the single point statistics of scalar field from a single realization of the solution of the equation \eqref{eq:advectionDiffusionNonDimension} and vice versa. Here, the single-point statistics, namely the moment of the random scalar field at point $(x,y)$, are $\left\langle T^{N}(x,y,t) \right\rangle=\mathbf{\Psi}_N (\mathbf{x},\mathbf{y},t)$ , where all components of $\mathbf{x}, \mathbf{y}$ are $x,y$, namely $x=x_{1}=x_{2}=...=x_{N}, y=y_{1}=y_{2}=...=y_{N}$. By the Sherman-Morrison formula \cite{sherman1950adjustment}, $\Lambda_{1}^{-1}=(\lambda^{(2)}-\lambda^{(1,1)})^{-1}\left( I- \frac{\lambda^{(1,1)} \mathbf{e}^{\text{T}}\mathbf{e} }{\lambda^{(2)}+ (N-1)\lambda^{(1,1)}} \right)$, and by the matrix determinant lemma, $\text{det} (\Lambda)= (\lambda^{(2)}-\lambda^{(1,1)})^{N}\left(1+ \frac{N \lambda^{(1,1)}}{\lambda^{(2)}-\lambda^{(1,1)}} \right)$. The equation \eqref{eq:NpointCorrelationShear} leads to the formula of $N$-th moment
\begin{equation}
\begin{array}{rl}\label{eq:NMoment}
\left\langle T^{N}(x,y,t) \right\rangle=&  \frac{\exp \left( - \frac{Nx^{2}}{2t (\lambda^{(2)}-\lambda^{(1,1)})} \left( 1- \frac{N\lambda^{(1,1)}}{\lambda^{(2)}+ (N-1)\lambda^{(1,1)}} \right) \right) }{(2\pi (\lambda^{(2)}-\lambda^{(1,1)}))^{\frac{N}{2}} \sqrt{1+ \frac{N \lambda^{(1,1)}}{\lambda^{(2)}-\lambda^{(1,1)}}}} \left( \int\limits_0^1 \hat{T}_{0} (0,y)\mathrm{d}y \right)^{N}          +\mathcal{O} (t^{-\frac{N+2}{2}} ). \\
\end{array}
\end{equation}
This formula \eqref{eq:NMoment} shows that $\lambda^{(1,1)}$ and $\lambda^{(2)}$ fully determines $T^N (x,y,t)$. Conversely, once $T^{N} (x,y,t)$ is known, a simple algebraic calculation yields the values of $\lambda^{(2)},\lambda^{(1,1)}$,
\begin{equation}
\begin{aligned}
\lambda^{(1,1)}= &\frac{1}{2\pi t} \sqrt{\left( \frac{\left\langle T^0 (0,0,t) \right\rangle }{  \left\langle T (0,0,t) \right\rangle} \right)^{4}- \left( \frac{\left\langle T^0 (0,0,t) \right\rangle }{\left\langle T^{2} (0,0,t) \right\rangle} \right)^{2}} \\
\lambda^{(2)}= &\frac{1}{2\pi t} \left( \frac{\left\langle T^0 (0,0,t) \right\rangle }{  \left\langle T (0,0,t) \right\rangle} \right)^{2} \\
\end{aligned}
\end{equation}
To show the ergodicity of this problem, it is enough to show we can compute the $\lambda^{(2)}, \lambda^{(1,1)}$ by the spatial and temporal average of a single realization of the random field. To do that, we first review the Aris moment. 

An alternative approach to study the enhanced dispersion induced by the shear flow is using the Aris moments. Aris showed in \cite{aris1956dispersion} that one could write down a recursive system of partial differential equations for the spatial moments of the tracer $T$. The streamwise moment and full moment are defined as:
\begin{equation}\label{eq:ArisMomentLongtime12}
\begin{aligned}
T_{n} (y,t) = & \int\limits_{-\infty}^{\infty}x^{n}T (x,y,t)\mathrm{d}x\\
\bar{T}_{n}=  & \int\limits_{0}^{1}T_{n} (y,t)\mathrm{d} y\\
\end{aligned}
\end{equation}
The first two full moments have the following long time asymptotic expansions
\begin{equation}\label{eq:ArisEffectiveDiffusion}
\begin{aligned}
\bar{T}_{1}=& \mathrm{Pe}\int\limits_0^{t}\bar{v}(\xi(s))\mathrm{d}s+\mathcal{O} (e^{-\pi^{2}}t)\\
\bar{T}_2  - \bar{T}_{1}^2= &  2\kappa_{\mathrm{eff}}t+\mathcal{O} (1)\\
\end{aligned}
\end{equation}
Using the ergodicity of $\xi(t)$ and the theorem \ref{thm:WindModelApproximation}, we have
\begin{equation}
\begin{aligned}
\frac{\lambda^{(2)}-\lambda^{(1,1)}}{2}= &\lim\limits_{t\rightarrow \infty} \frac{\bar{T}_2  - \bar{T}_{1}^2}{2t}, \\
  \lambda^{(1,1)}=& \lim\limits_{A\rightarrow \infty} \frac{1}{A}\int\limits_0^{A} \left( \frac{\partial T}{\partial t} (s) \right)^{2} \mathrm{d} s.\\
\end{aligned}
\end{equation}
where the second line  holds for a multiplicatively separable function $v (y,\xi(t))=u(y)\xi(t)$. For the non-multiplicative case, similar results hold. Hence, with the knowledge of a single realization of the random scalar field $T (x,y,t)$, we can compute the $\lambda^{(2)}, \lambda^{(1,1)}$ and all assemble moments $\left\langle T^{N} (x,y,t) \right\rangle$ in turn.  

\subsection{ Long time Asymptotic expansion of  OU process related time integral}
The effective diffusivity derived by the Aris moment approach and by theorem \ref{thm:WindModelApproximation} must be identical.
By solving the recursive equation of $T_n$ with the flow $v(y,\xi(t))=u(x) \xi (t)$, we derive the formula of the first and second Aris moments in the appendix \ref{sec:ArisMomentUXi}.  Then, the equation \eqref{eq:ArisMomentLongtime12} leads to the relation:
\begin{equation}
\begin{aligned}
 \kappa_{\mathrm{eff}}=& 1 + \lim\limits_{t\rightarrow \infty} \frac{\mathrm{Pe}^{2}}{t}  \sum\limits_{n=1}^{\infty}\left(  \int\limits_{0}^{1}u(y)\cos n\pi y \mathrm{d}y\right)^{2} \int\limits_{0}^{t} e^{-n^2\pi ^2  s}\xi (s)\int\limits _0^se^{n^2 \pi ^2 \tau} \xi(\tau)\mathrm{d} \tau \mathrm{d} s\\
\end{aligned}
\end{equation}
By the conclusion $\kappa_{\mathrm{eff}}= \frac{\lambda^{(2)}-\lambda^{(1,1)}}{2}$, we have the following relation
\begin{equation}\label{eq:OUIntegral}
\begin{aligned}
\frac{\mathrm{Pe}^{2}}{t}  \sum\limits_{n=1}^{\infty}\left(  \int\limits_{0}^{1}u(y)\cos n\pi y \mathrm{d}y\right)^{2} \int\limits_{0}^{t}\int\limits _0^s  e^{-n^2 \pi ^2 (s-\tau)} \xi(\tau)\xi (s) \mathrm{d} \tau \mathrm{d} s=&\frac{\lambda^{(2)}-\lambda^{(1,1)}}{2}-1 +\mathcal{O} (t^{-1})\\
\end{aligned}
\end{equation}
This relation provides a bunch of novel long time asymptotic expansions of OU process dependent integrals. For example, let $u (y)= \cos n \pi y$, we have
\begin{equation}\label{eq:OUIntegralExample1}
\begin{aligned}
I:=\frac{1}{t} \int\limits_{0}^{t} e^{-n^{2}\pi ^2  s}\xi (s)\int\limits _0^se^{n^{2}\pi ^2 \tau} \xi(\tau)\mathrm{d} \tau \mathrm{d}s&=\frac{1}{2}-\frac{\pi ^2 n^2}{2 \left(\gamma +\pi ^2 n^2\right)}+\mathcal{O} (t^{-1})\\
\end{aligned}
\end{equation}
In statistics, one interesting problem is to estimate the parameter $\gamma$ based on  (discrete or continuous) observations of $\xi (t), t \in [0,A]$ as $A\rightarrow \infty$ when $\gamma$ is unkown. The equation \eqref{eq:OUIntegralExample1} suggest an estimator of $\gamma$:
\begin{equation}
\begin{aligned}
\gamma &= \frac{2 n^{2}\pi ^2 I}{1-2 I}\\
\end{aligned}
\end{equation}
where $I$ denotes the left hand side of equation \eqref{eq:OUIntegralExample1}. One may choose suitable $u(y)$ to build a better estimator of $\gamma$ from relation \eqref{eq:OUIntegral}.

\section{Scalar Intermittency}
\label{sec:scalarIntermittency}
Now we switch our attention to the long time limiting PDF of the random scalar field.  The solution of the effective equation \eqref{eq:WindModel} have the same PDF as the original equation \eqref{eq:advectionDiffusionNonDimension} at long time. Unlike the original equation, the effective equation  has an explicit expression. Due to those properties the effective advection-diffusion equation  is a powerful tool for computing the long time limiting PDF. In this section, we will focus on the flow $v(y,\xi(t))=u(x)\xi(t)$ for three classes of initial data:1) deterministic initial data, 2) square integrable Gaussian random initial data, 3) wave function with a Gaussian random amplitude. 

\subsection{Deterministic Initial Data}\label{sec:DeterministicInitialData}
When the initial data is a deterministic integrable function, the long time asymptotic expansion of the solution of equation \eqref{eq:WindModel} is
\begin{equation}\label{eq:invariantMeasureDSol}
\begin{aligned}
T(x,y,t)= & \int\limits_{0}^{1}\hat{T}_0(0,y) \mathrm{d} y  \frac{1}{\sqrt{4\pi\kappa_{\mathrm{eff}} t}} \exp \left( - \frac{ \left( x- \mathrm{Pe}\bar{u} \int\limits_0^t \xi(s) \mathrm{d} s \right)^2}{4\kappa_{\mathrm{eff}}t}  \right)+O \left( \frac{1}{t^{\frac{3}{2}}} \right)  \\
\end{aligned}
\end{equation}
To explore the invariant measure of $T (x,t)$, we consider the rescaling of $T (x,t)$, 
\begin{equation}
\begin{aligned}
\tilde{T} (x,y,t)=\frac{\sqrt{4\pi\kappa_{\mathrm{eff}} t}}{\int\limits_{0}^{1}\hat{T}_0(0,y) \mathrm{d} y} T (x,y,t)=  \exp \left( - \frac{ \left( x- \mathrm{Pe}\bar{u} \int\limits_0^t \xi(s) \mathrm{d} s \right)^2}{4\kappa_{\mathrm{eff}}t}  \right)+O \left( \frac{1}{t} \right) \\
\end{aligned}
\end{equation}
From the above equation, we see that every point in the domain has the same leading order of the long time asymptotic expansion. Without loss of generality, we focus on the scalar at point $x=0, y=0$, $\tilde{T}(0,0,t)$. Thus, the ensuing probability density function is
\begin{equation}\label{eq:invariantMeasureD1}
\begin{aligned}
f_{\tilde{T}} (z)=\frac{z^{\frac{1}{\beta }-1}}{\sqrt{-\pi \beta  \log (z)}}\quad z\in[0,1]
\end{aligned}
\end{equation}
where $\beta= \frac{\mathrm{Pe}^2 \bar{u}^2 v(t)}{2 t \kappa _{\text{eff}}} =\frac{\mathrm{Pe}^2 \bar{u}^2 }{2  \kappa _{\text{eff}}}+O (t^{-1})$ and $v(t)$ is the variance of $\int\limits_0^t \xi (s)\mathrm{d}s$. $f_{\tilde{T}} (z)$ always has the logarithmic singularity at $z=1$. It is continuous at $z=0$ when $\beta\leq 1$, and singular when $\beta>1$ (see figure \ref{fig:DeterministicPdf}). Article \cite{camassa2020persisting} reports that as $\mathrm{Pe}$ increases, $f_{\tilde{T}} (z)$ changes from negatively-skewed to positively-skewed. Our formula for the invariant measure here quantitatively verifies that conclusion.

\begin{figure}
  \centering
  \subfigure{
    \includegraphics[width=0.8\linewidth]{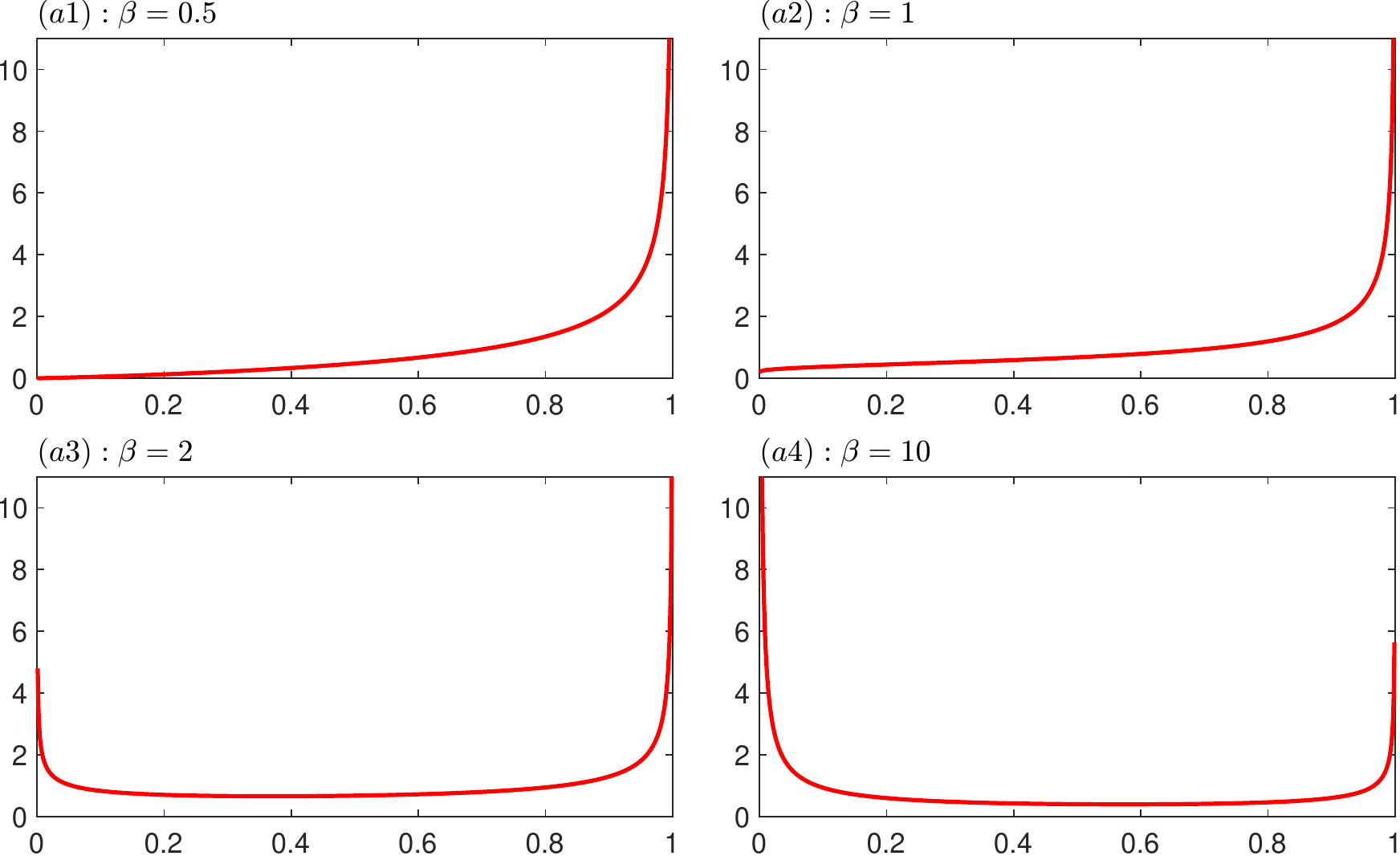}
  }
  \hfill
  \caption[]
  { The invariant measure $f_{\tilde{T}} (z)$  in equation \eqref{eq:invariantMeasureD1} for different parameters $\beta$. $f_{\tilde{T}} (z)$ changes from negatively-skewed to positively-skewed as $\beta$ increases. }
  \label{fig:DeterministicPdf}
\end{figure}

\begin{remark}
Since the formula of $N$-th moment is available  in this case, we can also derive the  invariant measure \eqref{eq:invariantMeasureD1} by the Laplace transformation based reconstruction method described in the section 3.4 of \cite{bronski2007explicit}. By equation \eqref{eq:NMoment}, we have
\begin{equation}
\begin{aligned}
  \left\langle T^N(0,0,t) \right\rangle=&\frac{1}{(4 \pi t \frac{\lambda^{(2)}-\lambda^{(1,1)}}{2})^{\frac{N}{2}}}\left( \int\limits_{0}^{1}\hat{T}_0(0,y) \mathrm{d} y \right)^{N} \frac{1}{\sqrt{1+  \frac{N\lambda^{(1,1)}}{\lambda^{(2)}-\lambda^{(1,1)}}}}+O(\frac{1}{t^{\frac{N+2}{2}}})\\
\end{aligned}
\end{equation}
 After rescaling, we have 
\begin{equation}
\begin{aligned}
  \left\langle \tilde{T}^N \right\rangle =& \frac{1}{\sqrt{N \beta +1}}+O(t^{-1})\\
\end{aligned}
\end{equation}
where $\beta=\frac{\lambda^{(1,1)}}{\lambda^{(2)}-\lambda^{(1,1)}}=\frac{\mathrm{Pe}^2 \bar{u}^2 }{2  \kappa _{\text{eff}}}$ which is equivalent to previous definition of $\beta$. Define the moment function as $\mu(s)=  \left\langle \tilde{T}^s \right\rangle = \frac{1}{\sqrt{s \beta+1}}+O(t^{-1})$ by extending the $N$-th moment formula from the integer domain to the complex domain. Once the moment function of $\tilde{T}$ is determined, we can compute the distribution of $\tilde{T}$ by the formula from \cite{bronski2007explicit}:
\begin{equation}
\begin{aligned}
 f(\xi)=& \frac{\mathscr{L}^{-1}(\mu(s))(-\ln \xi)}{\xi} \\
\end{aligned}
\end{equation}
where $\mathscr{L}$ denotes the Laplace transformation. By  inverse Laplace transformation, we derive the invariant measure \eqref{eq:invariantMeasureD1} again. 
This method requires the analytic formula of moment function and its inverse Laplace transformation. One may resort the effective equation approach when those information are not available.

\end{remark}

\begin{remark}
The invariant measure formula \eqref{eq:invariantMeasureD1} only used the single Fourier mode of the initial data $\int\limits_{0}^{1}\hat{T}(0,y)\mathrm{d} y$. One could obtain more accurate estimation of the rescaling factor and $\beta$ by using the whole information of initial data. For example, let's assume the initial data is $T_0 (x,y)= {\exp \left(-\frac{x^2}{2 s}\right)}/{\sqrt{2 \pi  s}}$. computing the long time asymptotic expansion of the solution of the wind model \eqref{eq:WindModel} without approximating $\int\limits_{0}^{1}\hat{T}(k,y)\mathrm{d} y$ by  $\int\limits_{0}^{1}\hat{T}(0,y)\mathrm{d} y$  yields 
\begin{equation}\label{eq:invariantMeasureD2}
\begin{aligned}
\tilde{T} (x,y,t)=\sqrt{2\pi s+4\pi\kappa_{\mathrm{eff}} t} T (x,y,t),& \quad 
\beta =\frac{2 \mathrm{Pe}^2 \bar{u}^2 v(t)}{4 t \kappa _{\text{eff}}+2 s} \\
\end{aligned}
\end{equation} 
 Formula \eqref{eq:invariantMeasureD1}, \eqref{eq:invariantMeasureD2} and $\beta$ lead to the same asymptotic result at long times. We implement the backward Monte-Carlo method described in the section 5 of \cite{camassa2020persisting}.  The simulation results shown in figure \ref{fig:invariantMeasure2} demonstrate that the formula \eqref{eq:invariantMeasureD2} is more accurate than formula \eqref{eq:invariantMeasureD1} at shorter times.    

\end{remark}

\begin{figure}
  \centering
    \includegraphics[width=5.5cm, height=3.4cm]{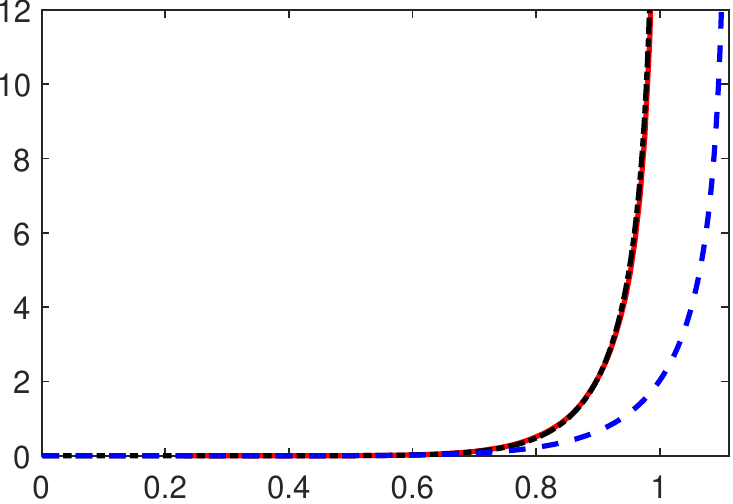}
  \hfill
  \caption[]
  {The invariant measure of the solution of equation \eqref{eq:advectionDiffusionNonDimension} with the flow $\mathrm{Pe}v(y,\xi (t))=y+\frac{1}{2}$ and initial condition $T_0(x,y)=\frac{1}{\sqrt{\pi}}e^{-x^2}$.  The red solid curve is the PDF from the numerical simulation at $t=1$. The blue dash curve is the  graph of equation \eqref{eq:invariantMeasureD1}. The back dot dash curve is the graph of equation \eqref{eq:invariantMeasureD2}. We use suitable rescaling factor for three of them such that the PDF from the simulation has the support $[0,1]$.}
  \label{fig:invariantMeasure2}
\end{figure}

\begin{remark}
With Gaussian random initial data and in the absence of impermeable boundaries, Vanden-Eijnden  \cite{vanden2001non} shows that the invariant measure of the scalar is independent the correlation time of the OU process. Whereas, in the presence of impermeable boundaries, our formula \eqref{eq:invariantMeasureD1} suggests that the correlation time has a significant impact on the shape of invariant measure, even the number of singularities it has. To get a deeper understanding about this, let's briefly review the free space problem with the deterministic initial data. In free space, we can derive the solution with the flow $v(y,\xi(t))=y\xi(t)$ and the initial condition $T_0 (x,y)=\delta(x)$ via method of characteristics:
\begin{equation}
\begin{aligned}
  T(x,y,t)= &\frac{\exp \left( \frac{- (x-y \xi(t))^2}{4 t (1+ \int\limits_{0}^{t}(\int\limits_{0}^{s}\xi(\tau)\mathrm{d} \tau)^{2}\mathrm{d}s)} \right)}
{\sqrt{4\pi t (1+ \int\limits_{0}^{t}(\int\limits_{0}^{s}\xi(\tau)\mathrm{d} \tau)^{2}\mathrm{d}s)}}  \\
\end{aligned}
\end{equation}
From this expression, we can see that this type of flow  will induce a time dependent random anomalous effective diffusivity $k_{\mathrm{eff}}=1+ \int\limits_{0}^{t}(\int\limits_{0}^{s}\xi(\tau)\mathrm{d} \tau)^{2}\mathrm{d}s$, where the second term on the right hand side is refer to as the $L^2$ norm of $ \int\limits_0^t \xi(s)\mathrm{d}s$. To obtain the invariant measure, we consider the rescaling $\tilde{T}=t\sqrt{4\pi} T$:
\begin{equation}
\begin{aligned}
\tilde{T}\sim & \left( \frac{\int\limits_{0}^{t}(\int\limits_{0}^{s}\xi(\tau)\mathrm{d} \tau)^{2}\mathrm{ds}}{t} \right)^{-\frac{1}{2}}\sim  \left( \int\limits_0^1B^{2} (s)\mathrm{d} s \right)^{-\frac{1}{2}} \quad \text{as } t\rightarrow \infty
\end{aligned}
\end{equation} 
where the second step follows the fact $\frac{1}{\sqrt{t}} \int\limits_0^{ts}\xi(\tau) \mathrm{d} \tau = B(s)+o(1)$ as $t \rightarrow \infty$. Hence, the invariant measure is only dependent on the $L^2$ norm of $B (t)$ and is independent on the correlation time of OU process. Unlike the quadratically growing variance in free space problem, the linearly growing variance allows the exponential function factor in the solution \eqref{eq:invariantMeasureDSol} to give a non vanishing contribution at long times.

\end{remark}


\subsection{Random Initial Data}
Although each realization of the random initial data is bounded, there is no uniform bound for all realizations of the initial data. This unboundedness makes the random field fall into the category of Hamburger moment problem rather than Hausdorff moment problem as the measure is not necessarily compactly supported. However, thanks to the incompressibility of the flow and the diffusion, the infinity norm of the random field decays at least algebraically. At a sufficiently large time and for the random initial data we studied in this subsection, the random field is almost surely bounded. The conclusion of the Hausdorff moment problem is valid thereafter. Another plausibility argument to fill this gap lies in the law of total probability, conditioning on a single realization of the initial data.  Denote $T, T'$ as the solution of the equation \eqref{eq:advectionDiffusionNonDimension} and the effective equation \eqref{eq:WindModel} respectively. Then we have
\begin{equation}
\begin{aligned}
f_{T}= & \int\limits_{g}^{}f_{T|T_{0}}(T|T_{0}=g) f_{T_{0}} (g)\mathrm{d} g\\
  \sim &\int\limits_{g}^{}f_{T'|T_{0}}(T'|T_{0}=g) f_{T_{0}} (g)\mathrm{d} g=f_{T'} \\
\end{aligned}
\end{equation}
where the first step follows the law of total probability by conditioning on the initial condition. The second step follows the theorem \ref{thm:WindModelApproximation}: $T$ and $T'$ have the same PDF at long time for the same deterministic initial condition. In addition, the law of total probability turn out to be a useful tool for studying the random initial data in free space problem\cite{camassa2008evolution}. By conditioning on a single realization of the flow, one takes the advantage of the Gaussianity of the initial data to compute the invariant measure more easily.



In this subsection, we will study the invariant measure of scalar field with the random wave initial data and square integrable spectral density, which has been studied in \cite{camassa2008evolution, bronski1997scalar} in the free space or with periodic boundary condition.
 
\subsubsection{Square Integrable Spectral Density}
In this section we will consider stratified initial data with
a square integrable spectral density which depends only upon
the spatial variable $x$,
\begin{equation}
\begin{aligned}
T_0= \int\limits_{-\infty}^{\infty}e^{\mathrm{i} h x}  \left| h \right|^{\frac{\alpha}{2}} \hat{\phi}_0(h)\mathrm{d}B (h) \quad \alpha>-1
\end{aligned}
\end{equation} where $\hat{\phi}_0 (h)$ denotes a rapidly decaying (large $h$) cut-off function satisfying $\hat{\phi}_0 (h)=\hat{\phi}^{*}_0 (-h), \hat{\phi}_0 (h) \neq 0$ and $\mathrm{d}B(h)$ denotes complex Gaussian white noise with the correlation function
\begin{equation}
\begin{aligned}
\left\langle \mathrm{d}B \right\rangle = 0,& \quad 
\left\langle \mathrm{d}B(h) \mathrm{d} B(\eta) \right\rangle =  \delta (h+\eta) \mathrm{d} h \mathrm{d} \eta\\
\end{aligned}
\end{equation} 
The spectral parameter $\alpha$ appearing in the initial data is introduced to adjust the excited length scales of the initial scalar field, with increasing $\alpha$ corresponding to initial data varying on smaller scales. 
It is enough to derive the long time invariant measure of the solution of equation \eqref{eq:WindModel}. The solution with this type of initial condition can be obtained by the Fourier transformation and method of characteristic.
The Fourier transformation yields
\begin{equation}
\begin{aligned}
\hat{T}_t- \mathrm{i} k \bar{u} \xi(t) \hat{T} &=- \kappa_{\mathrm{eff}} k^2 \hat{T} \\
  \hat{T}&= \hat{T}_{0}\exp(\mathrm{i} k \bar{u} \int\limits_0^t \xi(s) \mathrm{d} s - \kappa_{\mathrm{eff}} k^{2} t) \\
  \hat{T}&= 2 \pi \int\limits_{-\infty}^{\infty} \delta(h+k) \left| h \right|^{\frac{\alpha}{2}} \hat{\phi}_0(h) \mathrm{d}B(h) \exp(\mathrm{i} k \bar{u} \int\limits_0^t \xi(s) \mathrm{d} s - \kappa_{\mathrm{eff}} k^{2} t) \\
\end{aligned}
\end{equation}
Then the inverse Fourier transformation yields
\begin{equation}
\begin{aligned}
T(x)= & \int\limits_{-\infty}^{\infty}e^{\mathrm{i} h x} \left| h \right|^{\frac{\alpha}{2}} \hat{\phi}_0(h)  \exp(-\mathrm{i} h \bar{u} \int\limits_0^t \xi(s) \mathrm{d} s - \kappa_{\mathrm{eff}} h^{2} t)\mathrm{d}B(h) \\
\end{aligned}
\end{equation}
The leading order of the long time asymptotic expansion of the solution is independent of $x$. Without loss of generality, we focus on the solution at $x=0$, namely,
\begin{equation}
\begin{aligned}
T(0)= & \int\limits_{-\infty}^{\infty} \left| h \right|^{\frac{\alpha}{2}} \hat{\phi}_0(h)  \exp(-\mathrm{i} h M \int\limits_0^t \xi(s) \mathrm{d} s - \kappa_{\mathrm{eff}} h^{2} t)\mathrm{d}B(h) \\
\end{aligned}
\end{equation} 
By the law of total probability, the PDF of $T(0)$ has the integral representation
\begin{equation}
\begin{aligned}
f_{T}= & \int\limits_{-\infty}^{\infty}f_{T|\eta}(T|\eta=h) f_{\eta} (h)\mathrm{d} h\\
\end{aligned}
\end{equation}
where $\eta= \int\limits_0^t \xi(s) \mathrm{d} s $. Notice that $f_{T|\eta} \sim \mathcal{N} (0, \int\limits_{-\infty}^{\infty} \left| h \right|^{\alpha} \hat{\phi}_0^2 (h)\exp( - \kappa_{\mathrm{eff}} k^2 t) \mathrm{d} k )$ and $\eta \sim \mathcal{N} (0, t+\frac{ e^{-\gamma t}-1}{\gamma})$. Hence, the PDF of $T(0)$ independent of $\int\limits_0^t \xi(s) \mathrm{d} s$ and is a Gaussian random variable with variance $\int\limits_{-\infty}^{\infty} \left| h \right|^{\alpha} \hat{\phi}_0^2 (h)\exp( - \kappa_{\mathrm{eff}} h^2 t) \mathrm{d} h$. This conclude holds for any stochastic process $\xi(t)$, which generalizes the conclusion for Gaussian white noise process in \cite{bronski1997scalar}.

\subsubsection{Random Wave Initial Data}
In this section we will study Gaussian random wave initial data possessing zero spatial mean. We assume that the Fourier transform of the initial temperature
profile is highly localized as a function of the transform variable $k$,
\begin{equation}
\begin{aligned}
\hat{T}_0 (k) = 2\pi \left( A \delta(k+a)+ A^* \delta(k-a) \right)
\end{aligned}
\end{equation}
where the asterisk denotes the complex conjugate, $A$ is a standard complex Gaussian random variable, that is, $\Re (A), \Im (A) \sim \mathcal{N} (0, \frac{1}{2})$ and $\Re (A), \Im (A)$ are independent. We assume $a^{2}t\ll 1$ so that the the ground state energy expansion based theorem \ref{thm:WindModelApproximation} applies.

 In this case, we have
\begin{equation}
\begin{aligned}
  \hat{T}(k,t)&= 2\pi\left(  A \delta(k+a)+ A^* \delta(k-a) \right) \exp(\mathrm{i} k \mathrm{Pe} \bar{u}\int\limits_0^t \xi(s) \mathrm{d} s - \kappa_{\mathrm{eff}} k^{2} t) \\
\end{aligned}
\end{equation}
The inverse Fourier transformation yields
\begin{equation}
\begin{aligned}
  T (x,t)&= \exp(- \kappa_{\mathrm{eff}} a^{2} t) \left(A\exp(-\mathrm{i}ax+\mathrm{i} a M \int\limits_0^t \xi(s) \mathrm{d} s )+ A^{*}\exp(\mathrm{i} a x-\mathrm{i} a \mathrm{Pe}\bar{u} \int\limits_0^t \xi(s) \mathrm{d} s)  \right)\\
  &=2 \exp(- \kappa_{\mathrm{eff}} a^{2} t) \Re(A) \cos(ax+ a \mathrm{Pe}\bar{u} \int\limits_0^t \xi(s) \mathrm{d} s)
\end{aligned}
\end{equation}
To explore the invariant measure of $T (x,t)$, we consider the rescaling of $T (x,t)$, 
\begin{equation}
\begin{aligned}
\tilde{T} (x,t)=\exp( \kappa_{\mathrm{eff}} a^{2} t) T (x,t)= 2 \Re(A) \cos(\eta) \\
\end{aligned}
\end{equation}
where $\eta=ax+ a \mathrm{Pe}\bar{u} \int\limits_0^t \xi(s)\mathrm{d} s$. We have 
\begin{equation}
\begin{aligned}
 \eta  \mod 2\pi &\sim U([0, 2\pi]) \quad t\rightarrow \infty\\
2\Re (A)\cos (\eta)|\eta &\sim \mathcal{N} (0, \cos^{2} (\eta))\\
\end{aligned}
\end{equation}
Hence, the leading order of the PDF's long time asymptotic expansion is independent of the spatial variable $x$. By the law of total probability, we have 
\begin{equation}
\begin{aligned}
f_{\tilde{T}} (z) = & \int\limits_0^{2\pi}f_{T|\eta} (\tilde{T}|\eta =h) f_{\eta} (h)\mathrm{d} h  =\frac{e^{-\frac{z^2}{4}} K_0\left(\frac{z^2}{4}\right)}{\sqrt{2} \pi ^{3/2}} \\
\end{aligned}
\end{equation}
where $K_{n } (z)$ is the modified Bessel function of the second kind. $K_0 (z)$ is singular at $z=0$. The tail of the $f_{\tilde{T}} (z)$ is
\begin{equation}
\begin{aligned}
f_{\tilde{T}} (z) = &e^{-\frac{z^2}{2}} \left(\frac{1}{\pi  z}-\frac{1}{2 \pi  z^3}+O\left(\frac{1}{z^{5}}\right)\right)  \quad  z\rightarrow \infty\\
\end{aligned}
\end{equation}
The variance and fourth moment are $ \left\langle \tilde{T}^2 \right\rangle =\frac{1}{2}, \left\langle \tilde{T}^4\right\rangle = \frac{9}{8}$
and then the kurtosis (flatness) is $9/2>3$, which suggests the distribution could be flatter than the Gaussian distribution. In fact, this PDF has a smaller tail than Gaussian  distribution. The comparison of the invariant measure $f_{\tilde{T}} (z)$ and the PDF of the initial condition $T(0,0,0)$ in figure \ref{fig:randomwave} shows that the invariant measure has the larger core and smaller tails than the Gaussian distribution.

Bronski and McLaughlin \cite{bronski1997scalar} studied the problem with a Gaussian white noise process $\xi (t)$, $\bar{u}=0$ and periodic boundary conditions, who showed the invariant measure is Gaussian at some time scale. We also can obtain this conclusion by the effective equation approach. When $\bar{u}=0$, $\eta$ is a deterministic value. Hence the $\tilde{T}$ becomes a Gaussian random variable.

\begin{figure}
  \centering
    \includegraphics[width=0.46\linewidth]{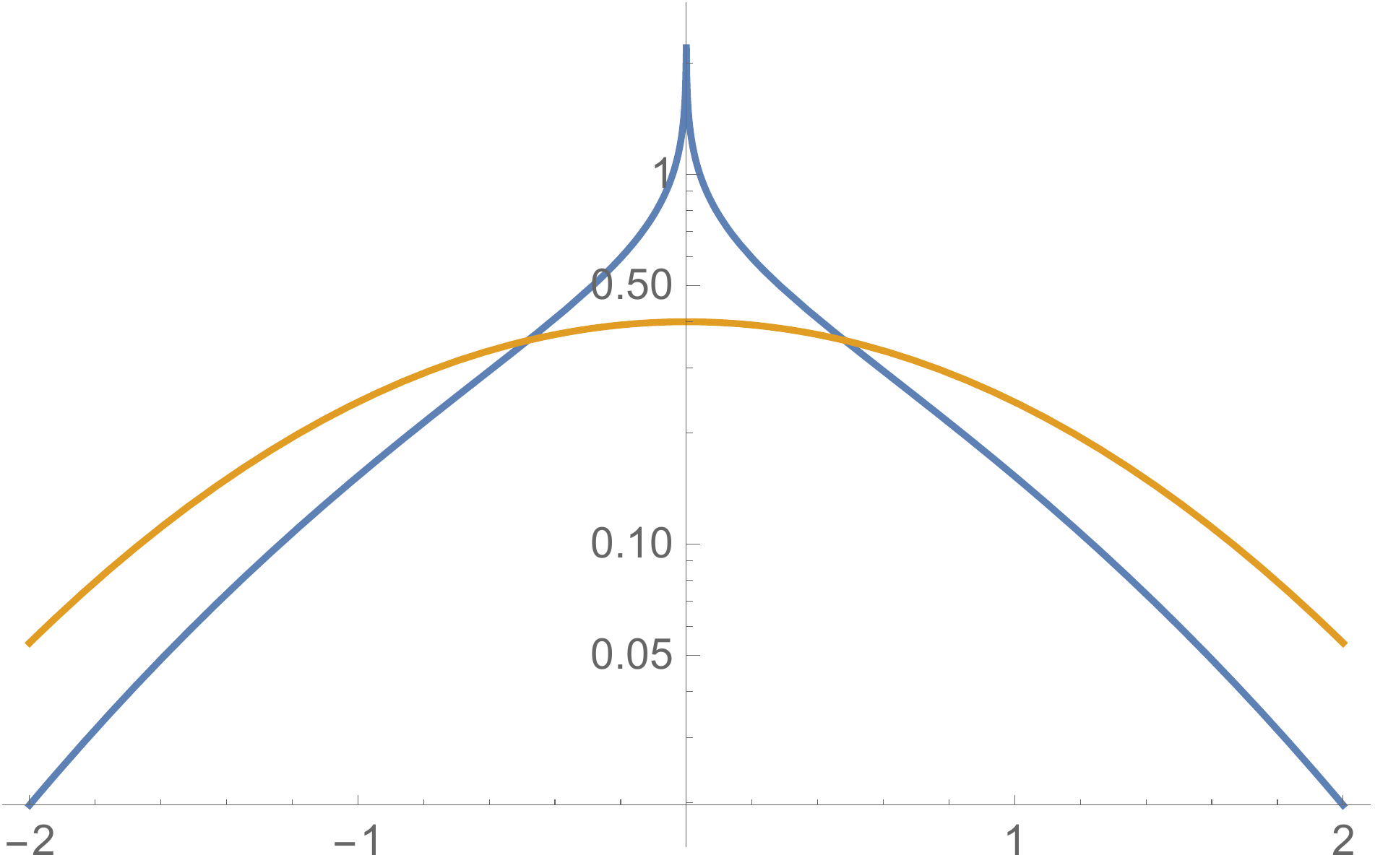}
  \caption[]
  {The semi-log plot of the distribution. The blue curve is the long time limiting PDF of $\tilde{T}(x,t)$ and the orange curve is the PDF of $T(0,0)$ which is a standard normal random variable $\mathcal{N} (0,1)$.  }
  \label{fig:randomwave}
\end{figure}

\section{Conclusion}
\label{sec:discuss}
We have studied a diffusing passive scalar in the presence of a OU dependent random shear flow in the presence of no-flux boundaries.  Long time asymptotic analysis of the closed moment equations produce simple formulae for the general $N$-point correlator.  We subsequently identified an effective advection-diffusion equation with random drift and deterministic enhanced diffusivity possessing the same exact long time moments as the full problem.  This advection-diffusion equation enjoys many properties, such as the centered second Aris moment being deterministic at all times.  Since the two equations have the exact same ensemble moments at long time, by the Hausdorff moment theorem, they have the same identical PDF at long time.  Consequently, given a single realization of the random velocity field, the centered second Aris moment of the original problem divided by $t$ must converge at long times to a deterministic constant set by the first two ensemble field moments.  Such ergodic are properties are rare in random partial differential equations, and here is particularly important when considering comparing the output of an experiment performed with a randomly moving wall (either normal or tangentially moving) with such a theory: It guarantees that one need only observe a {\em single} realization of the wall motion for the theory to be relevant at least in some measurable quantities.  Given these results, we additionally explored three different classes of initial data.  First, for deterministic initial data, we present formulae for the invariant measure.  Second, for square integrable random initial data, we show that the invariant measure will be Gaussian at long time.  Third, for waves with random amplitude, we show that the long time measures are non-Gaussian assuming the spatial average of the flow, $\bar u$, is non-zero, otherwise the limiting distribution will be Gaussian.  These results extend prior results of Bronski and McLaughlin \cite{bronski1997scalar} for more general random processes, and notably here for the case of random wave initial data, we compute the complete PDF (not just the flatness factors) for OU and white in time processes.  In previous work, we established results for the first three moments at long time \cite{camassa2020persisting}.  There we noted that in contrast with work in free-space by Vanden-Eijnden \cite{vanden2001non} where the PDF was observed to be independent of the correlation time at infinite time.  This distinction between free space and channel geometries we have extended in the present work to the full long time limiting PDF.  An interesting immediate direction involves computing the asymptotic corrections to the invariant measure.  The procedure employed by Bronski and McLaughlin \cite{bronski1997scalar} (who computed such corrections in the white noise limit through the fourth order ground state derivative) can be extended to these more general random processes.  Noteworthy, the OU case is considerably more involved as the odd derivatives do not vanish.

Future work will include considering an experimental campaign with the associated theoretical analysis. Our recent study \cite{ding2020enhanced} regarding the enhanced diffusion \cite{taylor1953dispersion} and third spatial Aris moment \cite{aris1956dispersion} induced by a periodically moving wall led to the development of an experimental framework of the model explored in this paper.  The computer controlled robotic arm we developed for the periodic study can be applied to the case of a randomly moving wall, such as the OU process $\xi(t)$, with suitable parameters for the fluid and the channel. The induced flow in the channel can be modeled by $y\xi(t)$. Hence, the tracers in the fluid satisfy the advection-diffusion equation \eqref{eq:advectionDiffusion}. The properties of the tracer's PDF can be predicted by the theory we developed here. Perhaps even more interesting will be considering cases in which the physical shear flow is not decomposed into a product of a function of space and a function of time, such as happens with the general nonlinear solutions to Stoke's second problem at finite viscosities.  We note that the general construction presented here for a OU dependent shear flow does not quite cover this case.  More involved analysis will clearly be needed to study these interesting configurations. Lastly, the random tangential motion of a non-flat wall will generate random non-sheared motions in the fluid.  We expect that applications of center manifold theory \cite{mercer1990centre,beck2015analysis,beck2020rigorous} may well be extendable to the case of random flows in such geometries.

\section{Acknowledge}
We acknowledge funding received from the following NSF Grant Nos.:DMS-1910824; and ONR Grant No: ONR N00014-18-1-2490.

\section{Appendix}

\subsection{Expansion of Eigenvalue and Eigenfunction}
\subsubsection{Impermeable Boundary Condition}\label{sec:PerturbationImpermeable}
After substituting the Taylor expansion of $\lambda_0 (\mathbf{k})$ and $\phi_{0} (\mathbf{k},\mathbf{y})$ with respect to $\mathbf{k}$ into the equation \eqref{eq:eigenvalueProblemOU}, we obtain the recursive relation of the coefficients in the expansion by comparing the coefficients of monomials of $\mathbf{k}$. We denote $\lambda^{\alpha}= \frac{\partial^{ \left| \alpha \right|}}{\partial \mathbf{k}^{\alpha}} \lambda (\mathbf{k}) |_{\mathbf{k}=\mathbf{0}}, \phi^{\alpha}= \frac{\partial^{ \left| \alpha \right|}}{\partial \mathbf{k}^{\alpha}} \phi (\mathbf{k},\mathbf{y}) |_{\mathbf{k}=\mathbf{0}}$, where $\alpha =(\alpha _{1},\alpha _{2},\ldots ,\alpha _{n})$  is a multi-index and $\left| \alpha \right|= \sum\limits_{i=1}^n a_i$. $\lambda_{0}^{(0)}, \phi_0^{(0)}$ satisfy the equation: 
\begin{equation}
\begin{aligned}
-\lambda_{0}^{(0)}\varphi_{0}^{(0)}=- \gamma z \frac{\partial \varphi_{0}^{(0)}}{\partial z} + \frac{\gamma}{2} \frac{\partial^2 \varphi_0^{(0)}}{\partial z^{2}} + \frac{\partial^2 \varphi_{0}^{(0)}}{\partial y^2},&\quad \left. \frac{\partial \varphi_{0}^{(0)}}{\partial y} \right|_{y= 0,1}=0 
\end{aligned}
\end{equation}
$\lambda_{0}^{(0)}=0, \phi_0^{(0)}=1$ are the solution. $\lambda_{0}^{(1)}, \phi_0^{(1)}$ satisfy the equation
\begin{equation}\label{eq:eigenvaluePN1O1}
\begin{aligned}
-\lambda_{0}^{(1)}=-\mathrm{i} \mathrm{Pe} v(y,\sqrt{\gamma}z) - \gamma z \frac{\partial \varphi_{0}^{(1)}}{\partial z} + \frac{\gamma}{2} \frac{\partial^2 \varphi_0^{(1)}}{\partial z^{2}}+ \frac{\partial^2 \varphi_0^{(1)}}{\partial y^2},&\quad \left. \frac{\partial \varphi_{0}^{(1)}}{\partial y} \right|_{y= 0,1}=0
\end{aligned}
\end{equation}
Fredholm alternative gives $\lambda_0^{(1)}=0$. Substituting the Hermite polynomial series representations $v (y,\sqrt{\gamma}z)= \sum\limits_{n=0}^{\infty}a_n (y,\sqrt{\gamma}) H_n (z), \phi_{0}^{(1)} (y,z)= -\mathrm{i} \mathrm{Pe}\sum\limits_{n=0}^{\infty}b_n (y,\sqrt{\gamma}) H_n (z)$ into the equation \eqref{eq:eigenvaluePN1O1}, where $H_n (z)$ is the $n$-th Hermite polynomial, gives the equation of $a_{n} (y,\sqrt{\gamma}), b_{n} (y,\sqrt{\gamma})$:
\begin{equation}
\begin{aligned}
 a_n -n \gamma b_n+ \frac{\partial^2 b_{n}}{\partial y^2}=0,&\quad   \frac{\partial b_{n}}{\partial y}|_{y= 0,1}=0 
\end{aligned}
\end{equation}
where we omit the argument $\sqrt{\gamma}$ in $a_{n},b_{n}$ to shorten the formula. We also introduce the inverse operator  $b (y) =\left( -\Delta+\lambda \right)^{-1}a (y)$ which maps the function $a (y)$ to the solution of the Helmholtz equation
\begin{equation}
-\frac{\partial^2 b (y)}{\partial y^2} + \lambda b (y)= a (y), \quad \left. \frac{\partial b}{\partial y} \right|_{ y=0,1}= 0
\end{equation}
and $b (y)$ has the integral representation
\begin{equation}\label{eq:HelmholtzSol1D}
\begin{aligned}
b (y)= &\frac{1}{\sqrt{\lambda }}\left( \frac{\cosh \left(\sqrt{\lambda } y\right) \int_0^1 a(s) \cosh \left(\sqrt{\lambda } (1-s)\right) \mathrm{d} s}{\sinh \left(\sqrt{\lambda } \right)}\right.\\
&\hspace{1cm} \left. -\int_0^y a(s) \sinh \left(\sqrt{\lambda } (y-s)\right) \mathrm{d} s \right) \\
b (y)=&-\int\limits_{0}^y \int\limits_{0}^{y_1}a(y_{2})\mathrm{d}y_2\mathrm{d}y_1\; \text{if}\; \lambda=0, \bar{a}=0\\
\end{aligned}
\end{equation}
With this notation, we have $b_n=\left( n\gamma-\Delta \right)^{-1}a_{n}$.

 $\lambda_{0}^{(2)}, \phi_0^{(2)}$ satisfy the equation
\begin{equation}\label{eq:eigenvaluePN1O2}
\begin{aligned}
-\lambda_{0}^{(2)}+2=-2\mathrm{i} \mathrm{Pe} v(y,z) \phi_0^{(1)}- \gamma z \frac{\partial \varphi_{0}^{(2)}}{\partial z} + \frac{\gamma}{2} \frac{\partial^2 \varphi_0^{(2)}}{\partial z^{2}}+ \frac{\partial^2 \varphi_0^{(2)}}{\partial y^2},&\quad \left. \frac{\partial \varphi_0^{(2)}}{\partial y} \right|_{y= 0,1}=0 
\end{aligned}
\end{equation}
Fredholm alternative gives
\begin{equation}
\begin{aligned}
\lambda_0^{(2)}= &2+ \left\langle2\mathrm{i} \mathrm{Pe} v(y,z) \phi_0^{(1)}  \right\rangle\\
=&2+2\mathrm{Pe}^{2}\sum\limits_{n=0}^{\infty} n! 2^{n}\int\limits_{0}^{1}a_n (y)\left( n\gamma-\Delta \right)^{-1}a_{n} (y)\mathrm{d} y \\
\end{aligned}
\end{equation} 
where the second step follows the identity $\frac{1}{\sqrt{\pi}}\int\limits_{-\infty}^{\infty} H_n^2 (z) e^{-z^2}=n!2^n$.  $\lambda_{0}^{(1,1)}, \phi_0^{(1,1)}$ satisfy the equation:
\begin{equation}\label{eq:eigenvaluePN1O11}
\begin{aligned}
-\lambda_{0}^{(1,1)}&=-\mathrm{i} \mathrm{Pe}\left( v(y_{1},z) \phi_0^{(0,1)}+ v(y_{2},z) \phi_0^{(1,0)}  \right)- \gamma z \frac{\partial \varphi_{0}^{(1,1)}}{\partial z} + \frac{\gamma}{2} \frac{\partial^2 \varphi_0^{(1,1)}}{\partial z^{2}}+\Delta_{2} \varphi_0^{(1,1)} \\
\left. \frac{\partial \varphi_0^{(1,1)}}{\partial y_{j}} \right|_{y_{j}=0,1}&=0 \quad j=1,2 
\end{aligned}
\end{equation}

Fredholm alternative gives:
\begin{equation}
\begin{aligned}
\lambda_0^{(1,1)}= & 2\left\langle \mathrm{i} \mathrm{Pe} v(y_{1},z) \phi_0^{(0,1)} (y_{2}),1 \right\rangle=\frac{2\mathrm{Pe}^{2}}{\gamma}\sum\limits_{n=1}^{\infty}  (n-1)! 2^{n}\left( \int\limits_{0}^{1}a_n (y)\mathrm{d} y \right)^{2} \\
\end{aligned}
\end{equation}
where the second step follows the series representation of $\int\limits_{0}^{1}\phi_0^{1} (y,z)\mathrm{d}y$:
\begin{equation}
\begin{aligned}
\int\limits_{0}^{1}\phi_0^{1} (y,z)\mathrm{d}y = &-\mathrm{i} \mathrm{Pe} \sum\limits_{n=1}^{\infty}\frac{1}{\gamma n} \int\limits_{0}^{1}a_{n} (y)\mathrm{d} y H_{n} (z)\\
\end{aligned}
\end{equation}
Alternative expression of $\lambda_0^{(1,1)}$ is available from the integral representation of $\int\limits_{0}^{1}\phi_0^{1} (y,z)\mathrm{d}y$:
\begin{equation}
\begin{aligned}
\int\limits_{0}^{1}\phi_0^{1} (y,z)\mathrm{d}y = & \frac{-2\mathrm{i} \mathrm{Pe}}{\gamma}(\int\limits_0^ze^{z_{2}^{2}}\int\limits_{-\infty}^{z_{2}} e^{-z_{1}^{2}}\bar{v}(z_{1})\mathrm{d}z_{1}\mathrm{d}z_{2} - \frac{1}{\sqrt{\pi}} \int\limits_{-\infty}^{\infty} e^{-z^2} \int\limits_0^ze^{z_{2}^{2}}\int\limits_{-\infty}^{z_{2}} e^{-z_{1}^{2}}\bar{v}(z_{1})\mathrm{d}z_{1}\mathrm{d}z_{2}\mathrm{d} z)\\
\end{aligned}
\end{equation}
Hence, we have
\begin{equation}
\begin{aligned}
\lambda_0^{(1,1)}= &\frac{4\mathrm{Pe}^{2} }{\gamma} \int\limits_{-\infty}^{\infty} e^{z^2} \left( \int\limits_{-\infty}^z e^{-z^2}\bar{v} (z_{1})\mathrm{d}z_1 \right)^{2}\mathrm{d}z\\
\end{aligned}
\end{equation}

\subsubsection{Periodic Boundary Condition}\label{sec:PerturbationPeriodic}
Instead of the impermeable boundary condition, we consider the periodic boundary condition and periodic flow in this section. we still have $\lambda_{0}^{(0)}=0, \phi_0^{(0)}=1, \lambda_{0}^{(1)}=0$. Then, we need to solve the equation \ref{eq:eigenvaluePN1O1} with the periodic boundary conditions. Assuming $v (y,\sqrt{\gamma}z), \phi_{0}^{(1)} (y,z)$ have the same form of Hermite polynomial series representations, the coefficient $a_{n} (y), b_{n} (y)$ satisfy the equation
\begin{equation}
\begin{aligned}
  a_n -n \gamma b_n+ \frac{\partial^2 b_{n}}{\partial y^2}=0,&\quad 
  b_{n} (0)=b_{n} \left( 1 \right),&\quad  \frac{\partial b_{n}}{\partial y}(0)= \frac{\partial b_{n}}{\partial y}(1)
\end{aligned}
\end{equation}
We can also represent the solution as $b_n=\left( n\gamma-\Delta \right)^{-1}a_{n}$. Now the operator has a different integral representation
\begin{equation}\label{eq:HelmholtzSolPeriodic}
\begin{aligned}
\left( \lambda-\Delta \right)^{-1} a (y)= &\frac{ \sinh \left(\sqrt{\lambda } \left(y-\frac{1}{2}\right)\right)\int_0^1 a(s) \sinh \left(\sqrt{\lambda } (L-s)\right) \, ds}{2 \sqrt{\lambda } \text{sinh}\left(\frac{\sqrt{\lambda }}{2}\right) }\\
&+\frac{\cosh \left(\sqrt{\lambda } \left(y-\frac{1}{2}\right)\right) \int_0^1 a(s) \cosh \left(\sqrt{\lambda } (1-s)\right) \, ds}{2 \sqrt{\lambda } \text{sinh}\left(\frac{\sqrt{\lambda }}{2}\right) }\\
& -\frac{\int_0^y a(s) \sinh \left(\sqrt{\lambda } (y-s)\right) \, ds}{\sqrt{\lambda }} \\
\Delta^{-1}a (y) =&- \int\limits_{0}^y \int\limits_{0}^{y_1} a (y_2)\mathrm{d}y_{2}\mathrm{d} y_1 + y\int\limits_{0}^{1}a (y_2)\mathrm{d}y_2+ \int\limits_{0}^{1} \int\limits_{0}^{y_1} a (y_2)\mathrm{d}y_{2}\mathrm{d} y_1 - \int\limits_{0}^{1}a (y_2)\mathrm{d}y_2\\
\end{aligned}
\end{equation}

With the similar perturbation analysis, we have the same form of series representation of $\lambda_0^{(2)}, \lambda_0^{(1,1)}$ except a different definition of operator $\left( \lambda-\Delta \right)^{-1}$.

\subsection{Aris Moment for the flow $u (y )\xi (t)$}\label{sec:ArisMomentUXi}
In this section, we will derive the second centered Aris moment for the flow  $u (y )\xi (t)$ and line source initial data $T_0 (x,y)=\delta (x)$. The Aris moments defined in equation \eqref{eq:ArisMomentDef} satisfy the recursive relationship called Aris equation,
\begin{equation}\label{eq:ArisMomentDef}
\begin{aligned}
(\partial_t- \Delta)T_n&=   n (n-1)T_{n-2}+ n  \mathrm{Pe}u(y,z,t) T_{n-1},\\
\left. \frac{\partial T}{\partial \mathbf{n}} \right|_{\partial \Omega}=  0, &\qquad T_n(y,z,0)= \int\limits_{-\infty}^{\infty} x^n T_{0}(x,y,z) \mathrm{d} x,\\
\end{aligned}
\end{equation}
where $T_{-1}=0$. The full moments of $T$ are then obtained though the cross-sectional average of the moments $ \bar{T}_n =  \int\limits_{0}^{1}T_n\mathrm{d}y$. Applying the divergence theorem and boundary conditions gives the recursive relationship of full moments,
\begin{equation}\label{eq:averArisMomentDef}
\begin{aligned}
\frac{\mathrm{d}  \bar{T}_{n} }{\mathrm{d} t}= &  n(n-1)  \bar{T}_{n-2}+ n  \mathrm{Pe}  \overline{u(y,z,t)  T_{n-1}} ,\\
 \bar{T}_n (0)=& \int\limits_{0}^{1}\int\limits_{-\infty}^{\infty} x^n T_{0}(x,y,z) \mathrm{d} x \mathrm{d}y.\\
\end{aligned}
\end{equation}

To compute the effective longitudinal  diffusivity, we need to compute the Aris moments $T_{0}, T_{1},  \bar{T}_2 $ in turn.
When $n=0$, the equation \eqref{eq:ArisMomentDef} becomes:
\begin{equation}
\frac{\partial T_0}{\partial t}- \frac{\partial^{2} T_{0}}{\partial y^2}= 0,\quad T_0(y,0)= 1, \quad \left. \frac{\partial T_{0}}{\partial y} \right|_{ y=0,1}= 0.
\end{equation}
The solution is $T_0=1$. When $n=1$, the equation \eqref{eq:ArisMomentDef} is:
\begin{equation}\label{eq:aris1}
\frac{\partial T_1}{\partial t}- \frac{\partial^{2} T_{1}}{\partial y^2}= \mathrm{Pe} u (y) \xi(t) T_0,\quad T_1(y,0)= 0, \quad \left. \frac{\partial T_{0}}{\partial y} \right|_{ y=0,1}= 0.
\end{equation}

The eigenfunction and eigenvalue of the Laplace operator on the cross section is $\lambda_{0}=0, \phi_{0}=1$, $\lambda_{n}=n^{2}\pi^{2}, \phi_{n}=\sqrt{2} \cos n \pi y, n\geq 1$ as the orthogonal basis. To compute $\bar{T}_2-\bar{T}_1^2$ with the flow $u (y)\xi (t)$ is equivalent to compute $\bar{T}_{2}$ with the flow $\left( u (y)-\int\limits_{0}^{1}u(y)\mathrm{d}y \right)\xi (t)$. Hence we will neglect the zero frequency mode in the expansion of $u(y)$. We assume the following expansion of $T_1$ and $u (y)\xi (t)$,
\begin{equation}
\begin{aligned}
  v(y,\xi(t)) =  \sum\limits_{n=1}^{\infty} \left\langle u,\phi_{n} \right\rangle \xi (t) \phi_{n},\quad   T_1 (y,t) =  \sum\limits_{n=1}^{\infty}  a_{n}(t) \phi_{n} \\
\end{aligned}
\end{equation}
$a_{i}(0)=0$ follows the initial condition $T_1(y,0)=0$. Substituting those expansions into the equation of \eqref{eq:aris1}, we obtain the equation of $a_n$ 
\begin{equation}
\begin{aligned}
 a'_{n}(t) +  \lambda_{n} a_{n}(t) &= \mathrm{Pe} \left\langle u,\phi_{n} \right\rangle \xi(t) \\
\end{aligned}
\end{equation}
The solution is 
\begin{equation}
\begin{aligned}
a_{n}&= \mathrm{Pe} \left\langle u, \phi_{n}\right\rangle  e^{-\lambda_n  t}\int _0^te^{\lambda_{n} s} \xi(s)ds\\
\end{aligned}
\end{equation}
$\bar{T}_{2}$ satisfies the equation
\begin{equation}
\begin{aligned}
\frac{\mathrm{d} \bar{T}_{2} }{\mathrm{d} t}&=  2 \bar{T}_{0}+  \mathrm{Pe} \xi (t)\overline{ u (y) T_1} \\
\end{aligned}
\end{equation}
With the initial condition $\bar{T}_{2}(0)= 0$, we have 
\begin{equation}
\begin{aligned}
\bar{T}_2(t)&= 2t +  2\mathrm{Pe}^{2} \sum\limits_{n=1}^{\infty} \left\langle u,\phi_{n} \right\rangle^{2} \int\limits_{0}^{t} e^{-\lambda_{n} s}\xi (s)\int\limits _0^se^{\lambda_{n} \tau} \xi(\tau)\mathrm{d} \tau \mathrm{d} s
\end{aligned}
\end{equation}
and
\begin{equation}
\begin{aligned}
\kappa_{\mathrm{eff}}&= 1 + \lim\limits_{t\rightarrow \infty} \frac{\mathrm{Pe}^{2}}{t} \sum\limits_{n=1}^{\infty} \left\langle u,\phi_{n} \right\rangle^{2} \int\limits_{0}^{t} e^{-\lambda_{n} s}\xi (s)\int\limits _0^se^{\lambda_{n} \tau} \xi(\tau) \mathrm{d} \tau \mathrm{d} s\\
\end{aligned}
\end{equation}

\bibliographystyle{elsarticle-harv}


\end{document}